\documentclass[12pt,reqno]{amsart}
\usepackage{amssymb}
\usepackage{amsfonts}
\usepackage{amsmath}
\usepackage{amsmath}
\usepackage{verbatim}
\usepackage{color}
\usepackage[shortlabels] {enumitem}
\usepackage[active]{srcltx}
\include{srctex}
\usepackage[
bookmarks=true,         
bookmarksnumbered=true, 
colorlinks=true,
pdfstartview=FitV,
linkcolor=blue, citecolor=blue, urlcolor=blue]{hyperref}
\usepackage[
	alphabetic,abbrev,msc-links]{amsrefs}

\usepackage{fullpage}

\allowdisplaybreaks[4]

\newtheorem{theorem}{Theorem}[section]

\newtheorem{conjecture}[theorem]{Conjecture}

\newtheorem{lemma}[theorem]{Lemma}

\newtheorem{proposition}[theorem]{Proposition}
\theoremstyle{remark}
\newtheorem{remark}[theorem]{Remark}

\newtheorem{definition}[theorem]{Definition}
\newtheorem{example}[theorem]{Example}

\let\Section=\section
\def\section{\setcounter{equation}{0}\Section}

\newcommand{\1}{\mathbf{1}}

\def\RR{\mathbb{R}}

\def\EE{\mathbb{E}}
\def\PP{\mathbb{P}}
\def\CC{\mathbb{C}}
\def\HH{\mathfrak{H}}
\def\Z{\mathcal{Z}}
\def\K{\mathcal{K}}

\def\U{\mathcal{U}}
\def\A{\mathcal{A}}
\def\B{\mathcal{B}}

\def\css{{\mathcal S}}
\def\cee{{\mathcal E}}
\def\cmm{{\mathcal M}}
\def\caa{{\mathcal A}}
\def\cff{{\mathcal F}}
\def\cgg{{\mathcal G}}

\def\lt{\left}
\def\rt{\right}
\def\HH{\mathfrak{H}}
\def\supp{\mathrm{supp}\,}
\begin{document}

\title{Spatial asymptotic of the stochastic heat equation with compactly supported initial data}
\author{Jingyu Huang}
\author{Khoa L\^e}
\address{Department of Mathematics, University of Utah, Salt Lake City, Utah, 84112, USA}
\address{Department of Mathematical and Statistical Sciences, University of Alberta, 632 Central Academic, Edmonton , AB T6G 2R3, Canada}
\thanks{This work is partially supported by NSF Grant no. 0932078000}
\date{\today}
\begin{abstract}
	We investigate the growth of the tallest peaks of random field solutions to the parabolic Anderson models over concentric balls as the radii approach infinity. The noise is white in time and correlated in space. The spatial correlation function is either bounded or non-negative satisfying Dalang's condition. The initial data are Borel measures with compact supports, in particular, include Dirac masses. The results obtained are related to those of \cites{MR3098071,MR3474477} where constant initial data are considered. 
\end{abstract}
\maketitle
\section{Introduction} 
\label{sec:introduction}
	We consider the stochastic heat equation in $\RR^\ell$
	\begin{equation}\label{eqn:SHE}
		 \frac {\partial u}{\partial t} =\frac12\Delta u+u \dot{W}\, ,\quad u(0,\cdot)=u_0(\cdot)
	\end{equation}
	where  $t\ge 0$, $x\in \RR^\ell$ $(\ell\ge1)$ and $u_0$ is a Borel measure. Herein, $W$ is  a centered Gaussian field, which is white in time and it has a correlated spatial covariance. More precisely, we assume that the noise $W$ is described by a centered  Gaussian family  $W=\{ W(\phi), \phi \in C_c^\infty(\RR_+\times\RR^\ell)\}$, with covariance
	\begin{equation}\label{eqn:fouriercov}
		\EE [W(\phi)W(\psi)]=\frac1{(2 \pi)^\ell} \int_0^\infty\int_{\RR^\ell} \cff\phi(s,\xi)\overline{\cff\psi(s,\xi)}\mu(\xi) d \xi ds,
	\end{equation}
	where $\mu$ is non-negative measurable function and $\cff$ denotes the Fourier transform in the spatial variables. To avoid trivial situations, we assume that $\mu$ is not identical to zero. The inverse Fourier transform of $\mu$ is in general a distribution defined formally by the expression
	\begin{equation}\label{eqn:gamspec}
		\gamma(x)=\frac1{(2 \pi)^\ell}\int_{\RR^\ell} e^{i \xi  \cdot x}\mu(\xi)d \xi\,.
	\end{equation}
	If $\gamma$ is a locally integrable function, then it is non-negative definite and \eqref{eqn:fouriercov} can be written in Cartesian coordinates
	\begin{equation}\label{eqn:cov}
		\EE [W(\phi)W(\psi)]=\int_0^\infty\iint_{\RR^{2\ell}} \phi(s,x)\psi(s,y)\gamma(x-y)dxdyds\,.
	\end{equation}
	The following two distinct hypotheses on the spatial covariance of $W$ are considered throughout the paper.
	\begin{enumerate}[(H.\arabic*)]
		\item\label{Cov:bounded} $\mu$ is integrable, 
			that is $\int_{\RR^\ell}\mu(\xi)d \xi<\infty$. 
			In this case, the inverse Fourier transform of $\mu(\xi)$ exists and is a bounded continuous function $\gamma$. Assume in addition that $\gamma$ is $\kappa$-H\"older continuous  function at 0.
		\item\label{Cov:Dalang} $\mu$ satisfies the following conditions:
		\begin{enumerate}[(H.2a)]
			\item  The inverse Fourier transform of $\mu(\xi)$ is either the Dirac delta mass at 0 or a nonnegative  locally integrable function $\gamma$.
			\item\label{h2a} 
			\begin{equation}\label{k5}
				\int_{ \RR^\ell}\frac {\mu(\xi) }{1+ |\xi|^2}d \xi<\infty\,.
			\end{equation}
			\item\label{con:scaling} (Scaling) There exists $\alpha\in(0,2)$ such that $\mu(c \xi)=c^{\alpha-\ell}\mu(\xi)$ for all positive numbers $c$.
		\end{enumerate}
	\end{enumerate}
	Hereafter, we denote by $|\cdot|$ the Euclidean norm in $\RR^\ell$ and by $x\cdot y$ the usual inner product between two vectors $x,y$ in $\RR^\ell$. 
	Condition \ref{h2a} is known as Dalang's condition and is sufficient for existence and uniqueness of a random field solution.
	If $\gamma$ exists as a function, condition \ref{con:scaling} induces the scaling relation $\gamma(c x)=c^{-\alpha}\gamma(x)$ for all $c>0$. 

	Equation \eqref{eqn:SHE} with noise satisfying condition \ref{Cov:Dalang} was introduced by Dalang in \cite{Dal}. In \cite{HLN15}, for a large class of initial data, we show that equation \eqref{eqn:SHE} has a unique random field solution under the hypothesis \ref{Cov:Dalang}. Under hypothesis \ref{Cov:bounded}, we note that $\gamma$ may be negative, but proceeding as in \cite{Huang}, a simple Picard iteration argument gives the existence and uniqueness of the solution. In addition, in both cases, the solution has finite moments of all positive orders. We give a  few examples of  covariance structures which are usually considered in literatures.
	\begin{example}\label{expl:cov1}
		Covariance functions satisfying \ref{Cov:Dalang} includes the Riesz kernel $\gamma(x)=|x|^{-\eta}$, with $0<\eta <2\wedge \ell$, the space-time white noise in dimension one, where $\gamma =\delta_0$, the Dirac delta mass at 0, and the multidimensional fractional Brownian motion, where $\gamma(x)= \prod_{i=1}^\ell H_i (2H_i-1) |x^i|^{2H_i-2}$, assuming $\sum_{i=1}^\ell H_i >\ell-1$ and $H_i >\frac 12$ for $i=1,\dots, \ell$. Covariance functions satisfying \ref{Cov:bounded} includes $e^{-|x|^2}$ and the inverse Fourier transform of $|\xi|^2 e^{-|\xi|^2}$. 
	\end{example}
	Suppose for the moment that $\dot{W}$ is a space-time white noise and $u_0$ is a function satisfying
	\begin{equation}\label{eq:u0 lower upper bd}
	 	c\le u_0(x)\le C,\mbox{ for some positive numbers } c,C.
	\end{equation} 
	It is first noted in \cite{MR3098071} that there exist positive constants $c_1,c_2$ such that almost surely
	\begin{equation}\label{eqn:cuC}
		c_1\le \limsup_{R\to\infty}(\log R)^{-\frac23}{\log\sup_{|x|\le R}u(t,x)} \le c_2\,.
	\end{equation}
	Later Xia Chen shows in \cite{MR3474477} that indeed the precise almost sure limit can be computed, namely, 
	\begin{equation}\label{eqn:lim1}
		\lim_{R\to\infty}(\log R)^{-\frac23}{\log\sup_{|x|\le R}u(t,x)}=\frac34 \lt(\frac{2t}3\rt)^{\frac13}\quad \mathrm{a.s.}
	\end{equation}
	One of the key ingredients in showing \eqref{eqn:lim1} is the following moment asymptotic result
	\begin{equation}\label{eqn:lim2}
		\lim_{m\to\infty}m^{-3}\log\EE u(t,x)^m=\frac t{24}\,.
	\end{equation}
	Thanks to the scaling property of the space-time white noise, Xia Chen has managed to derive \eqref{eqn:lim2} from the following long term asymptotic result
	\begin{equation}\label{eqn:lim3}
		\lim_{t\to\infty}\frac1t\log\EE u(t,x)^m= \cee_m
	\end{equation}
	where the constant $\cee_m$ grows as $\frac1{24}m^3$ when $m\to\infty$. 

	Under condition \eqref{eq:u0 lower upper bd}, analogous results for other kinds of noises are also obtained in \cite{MR3474477}. More precisely, for noises satisfying \ref{Cov:bounded}
	\begin{equation}
		\lim_{R\to\infty}(\log R)^{-\frac12}\log \sup_{|x|\le R}u(t,x)=\sqrt{2\ell \gamma(0)t}\quad \mathrm{a.s.}\,,
	\end{equation}
	and for noises satisfying \ref{Cov:Dalang}, 
	\begin{equation}
			\lim_{R\to\infty}(\log R)^{-\frac2{4- \alpha}}\sup_{|x|\le R}\log u(t,x)=\frac{4- \alpha}2\ell^{\frac2{4- \alpha}} \lt(\frac{\cee_H(\gamma)}{2- \alpha}t \rt)^{\frac{2- \alpha}{4- \alpha}} \quad \mathrm{a.s.}\,,
	\end{equation}
	where the variational quantity $\cee_H(\gamma)$ is introduced in \eqref{eqn:cee}. 

	On the other hand, it is known that equation \eqref{eqn:SHE} has a unique random field solution under either \ref{Cov:bounded} or \ref{Cov:Dalang} provided that $u_0$ satisfies
	\begin{equation}\label{con:u0}
		p_t*|u_0|(x) <\infty\quad\forall t>0,x\in\RR^\ell\,.
	\end{equation}
	Hence, condition \eqref{eq:u0 lower upper bd} excludes other initial data of interests such as compactly supported measures. It is our purpose in the current paper to investigate the almost sure spatial asymptotic of the solutions corresponding to these initial data. 

	Upon reviewing the method in obtaining \eqref{eqn:lim1} described previously, one first seeks for an analogous result to \eqref{eqn:lim3} for general initial data. In fact, it is noted in \cite{HLN15} that for every $u_0$ satisfying \eqref{con:u0}, one has
	\begin{equation}\label{eqn:lim4}
		\lim_{t\to\infty}\frac1t\log\sup_{x\in\RR} \EE\lt(\frac{u(t,x)}{p_t*u_0(x)} \rt)^m=\cee_m,	
	\end{equation}
	where $\cee_m$ is a constant whose asymptotic as $m\to\infty$ is known. 
	It is suggestive from \eqref{eqn:lim4} that with a general initial datum, one should normalized $u(t,x)$ in \eqref{eqn:lim1} (and \eqref{eqn:lim2}) by the factor $p_t*u_0(x)$. Therefore, we anticipate the following almost sure spatial asymptotic result.
	\begin{conjecture}\label{conjecture}
		Assume that $u_0$ satisfies \eqref{con:u0}. Under \ref{Cov:bounded} we have
		\begin{equation}\label{coj:h1}
			\lim_{R\to\infty}(\log R)^{-\frac12} \sup_{|x|\le R}\lt(\log u(t,x)-\log p_t*u_0(x) \rt)=\sqrt{2\ell \gamma(0)t}\quad \mathrm{a.s.}
		\end{equation}
		Under \ref{Cov:Dalang}, we have
		\begin{equation}\label{coj:h2}
				\lim_{R\to\infty}(\log R)^{-\frac2{4- \alpha}}\sup_{|x|\le R}\lt(\log u(t,x)-\log p_t*u_0(x) \rt)=\frac{4- \alpha}2\ell^{\frac2{4- \alpha}} \lt(\frac{\cee_H(\gamma)}{2- \alpha}t \rt)^{\frac{2- \alpha}{4- \alpha}} \quad \mathrm{a.s.}
		\end{equation}	 
		In the particular case of space-time white noise, we conjecture that
		\begin{equation}\label{eqn:lim5}
			\lim_{R\to\infty}(\log R)^{-\frac23}{\sup_{|x|\le R}(\log u(t,x)-\log p_t*u_0(x))}={\frac34 \lt(\frac{2t}3\rt)^{\frac13} }\quad \mathrm{a.s.}
		\end{equation}
	\end{conjecture} 
	In the case of space-time white noise, note that if $u_0$ satisfies the condition \eqref{eq:u0 lower upper bd}, \eqref{eqn:lim5} is no different than \eqref{eqn:lim1}. On the other hand, if $u_0$ is a Dirac delta mass at $x_0$, \eqref{eqn:lim5} precisely describes the spatial asymptotic of $\log u(t,x)$: at large spatial sites, $\log u(t,x)$ is concentrated near a logarithmic perturbation of the parabola $-\frac{1}{2t} (x-x_0)^2$. More precisely, \eqref{eqn:lim5} with this specific initial datum reduces to
	\begin{equation}\label{eqn:lim6}
		\lim_{R\to\infty}(\log R)^{-\frac23} \sup_{|x|\le R}\lt(\log u(t,x)+\frac{(x-x_0)^2}{2t} \rt)={\frac34 \lt(\frac{2t}3\rt)^{\frac13}}\,.
	\end{equation}
	While a complete answer for Conjecture \ref{conjecture} (including \eqref{eqn:lim6}) is still undetermined, the current paper offers partial results, focusing on initial data with compact supports, especially Dirac masses. 
	To unify the notation, we denote 
\begin{equation}\label{eq:unify notation}
 	\bar \alpha=
 	\lt\{
 		\begin{array}{ll}
 			0&\mbox{if \ref{Cov:bounded} holds},
 			\\\alpha&\mbox{if \ref{Cov:Dalang} holds},
 		\end{array}
 	\rt.
 	\quad\mbox{and}\quad
 	\cee=
 	\lt\{
 		\begin{array}{ll}
 			\gamma(0)&\mbox{if \ref{Cov:bounded} holds},
 			\\\cee_H(\gamma)&\mbox{if \ref{Cov:Dalang} holds},
 		\end{array}
 	\rt.
\end{equation} 
	where the variational quantity $\cee_H(\gamma)$ is introduced below in \eqref{eqn:cee}. 
	For bounded covariance functions, we obtain the following result.
	\begin{theorem}\label{thm:H1Z}
		Assume that \ref{Cov:bounded} holds and $u_0=\delta(\cdot - x_0)$ for some $x_0\in\RR^\ell$.
		Then \eqref{coj:h1} holds.
	\end{theorem}	
	
For noises satisfying \ref{Cov:Dalang}, or for initial data with compact supports, the picture is less complete. 
	\begin{theorem}\label{thm:H23Z}
		Assume that $u_0$ is a non-negative measure with compact support and either \ref{Cov:bounded} or \ref{Cov:Dalang} holds. Then we have
\begin{equation}\label{lim:Zspatial}
			\limsup_{R\to\infty}(\log R)^{-\frac2{4- \bar{\alpha}}}\sup_{|x|\le R}\lt(\log u(t,x)-\log p_t*u_0(x) \rt)\leq\frac{4- \bar{\alpha}}2\ell^{\frac2{4- \bar{\alpha}}} \lt(\frac{\cee}{2- \bar{\alpha}}t \rt)^{\frac{2- \bar{\alpha}}{4- \bar{\alpha}}} \quad \mathrm{a.s.}
		\end{equation}		 		 
	\end{theorem}
	For initial data satisfying \eqref{eq:u0 lower upper bd}, the lower bound of \eqref{coj:h2} is proved in \cite{MR3474477} using a localization argument initiated from \cite{MR3098071}. In our situation, a technical difficulty arises in applying this localization procedure, 
	which leads to the missing lower bound in Theorem \ref{thm:H23Z}. A detailed explanation is given at the beginning of Subsection \ref{sub:low}. As an attempt to obtain the exact spatial asymptotics, we propose an alternative result which is described below. We need to introduce a few more notation.
		For each $\epsilon>0$, we denote 
	\begin{equation}
		\gamma_ \epsilon(x)=(2 \pi)^{-\ell}\int_{\RR^\ell}e^{- 2\epsilon|\xi|^2}e^{i \xi\cdot x}\mu(\xi)d \xi\,,
	\end{equation}
	which is a bounded non-negative definite function. Let $W_ \epsilon$ be a centered Gaussian field defined by
	\begin{equation}\label{eqn:Wep}
		W_ \epsilon(\phi)=W(p_{\epsilon}*\phi)
	\end{equation}
	for all $\phi\in C_{c}^{\infty}(\RR_+\times\RR^\ell)$. In the above, $p_ \epsilon=(2 \pi \epsilon)^{-\ell/2}e^{-|x|^2/(2 \epsilon)}$ and $p_{\epsilon}*\phi$ is the convolution of $p_{\epsilon}$ with $\phi$ in the spatial variables. The covariance structure of $W_ \epsilon$ is given by
	\begin{align}
		\EE [W_ \epsilon(\phi)W_ \epsilon(\psi)]&=\frac1{(2 \pi)^\ell} \int_0^\infty\int_{\RR^\ell} \cff\phi(s,\xi)\overline{\cff\psi(s,\xi)}e^{-2 \epsilon|\xi|^2}\mu( \xi) d \xi ds
		\nonumber\\&=\int_0^\infty\iint_{\RR^{2\ell}} \phi(s,x)\psi(s,y)\gamma_ \epsilon(x-y)dxdyds\label{Wep.cov}
	\end{align}	
	for all  $\phi,\psi\in C_c^{\infty}(\RR_+\times\RR^\ell)$.
	In other words, $W_ \epsilon$ is white in time and correlated in space with spatial covariance function $\gamma_ \epsilon$, which satisfies \ref{Cov:bounded}. Under condition \ref{con:scaling}, $\gamma_ \epsilon$ satisfies the scaling relation
	\begin{equation}\label{eqn:scalgamma}
			\gamma_ \epsilon(x)=\epsilon^{-\frac \alpha2}\gamma_1(\epsilon^{-\frac12} x)\quad \mbox{for all}\quad \epsilon>0,x\in\RR^\ell\,.
	\end{equation}
	Let $u_ \epsilon$ be the solution to equation \eqref{eqn:SHE} with $\dot{W}$ replaced by $\dot{W}_ \epsilon$. 
 It is expected that as $\epsilon\downarrow0$, $u_ \epsilon(t,x)$ converges to $u(t,x)$ in $L^2(\Omega)$ for each $(t,x)$, see \cite{ChenHuang} for a proof when the initial data is a bounded function. The following result describes spatial asymptotic of the family of random fields $\{u_\epsilon\}_{\epsilon\in(0,1)}$.
	\begin{theorem}\label{thm:uep}
	Assume that $u_0$ is a non-negative measure with compact support and either \ref{Cov:bounded} or \ref{Cov:Dalang} holds. Then
		\begin{equation}\label{lim:H2epup}
			\limsup_{R\to\infty}(\log R)^{-\frac2{4- \bar{\alpha}}}\sup_{|x|\le R,\epsilon\in(0,1)}\lt(\log u_ \epsilon(t,x)-\log p_t*u_0(x) \rt)
			\leq\frac{4- \bar{\alpha}}2\ell^{\frac2{4- \bar{\alpha}}} \lt(\frac{\cee}{2- \bar{\alpha}}t \rt)^{\frac{2- \bar{\alpha}}{4- \bar{\alpha}}}\quad \mathrm{a.s.}
		\end{equation} 
		If, in particular, $u_0=\delta(\cdot-x_0)$ for some $x_0 \in \RR^{\ell}$, then
		\begin{equation}\label{lim:H2ep}
			\lim_{R\to\infty}(\log R)^{-\frac2{4- \bar{\alpha}}}\sup_{|x|\le R,\epsilon\in(0,1)}\lt(\log u_ \epsilon(t,x) + \frac{(x-x_0)^2}{2t} \rt)
			=\frac{4- \bar{\alpha}}2\ell^{\frac2{4- \bar{\alpha}}} \lt(\frac{\cee}{2- \bar{\alpha}}t \rt)^{\frac{2- \bar{\alpha}}{4- \bar{\alpha}}}\quad \mathrm{a.s.}
		\end{equation} 		
	\end{theorem}
	Neither one of \eqref{coj:h2} and \eqref{lim:H2ep} is stronger than the other. While the result of Theorem \ref{thm:uep} relates to the solution of \eqref{eqn:SHE} indirectly, it is certainly interesting. In Hairer's theory of regularity structures (cf. \cite{MR3274562}), one first regularizes the noise to obtain a sequence of approximated solutions. The solution of the corresponding stochastic partial differential equation is then constructed as the limiting object of this sequence. From this point of view, \eqref{lim:H2ep} provides a unified characteristic of the sequence of approximating solutions $\{u_ \epsilon\}_{\epsilon\in(0,1)}$, which approaches the solution $u$ as $\epsilon\downarrow 0$. 
	The proof of \eqref{lim:H2ep} does not rely on localization, rather, on the Gaussian nature of the noise. This leads to a possibility of extending \eqref{lim:H2ep} to temporal colored noises, which will be  a topic for future research.

	The remainder of the article is structured as follows: In Section \ref{sec:preliminary} we briefly summarize the theory of stochastic integrations and well-posedness results for \eqref{eqn:SHE}.   In Section \ref{sec:variations} we introduce some variational quantities which are related to the spatial asymptotics. In Section  \ref{sub:feynman_kac_formula} we derive some Feynman-Kac formulas of the solution and its moments, these formulas play a crucial role in our consideration. In Section \ref{sec:moment_asymptotic_and_regularity} we investigate the high moment asymptotics and H\"older regularity of the solutions of \eqref{eqn:SHE} with respect to various parameters. The results in  Section \ref{sec:moment_asymptotic_and_regularity} are used to obtain upper bounds in \eqref{coj:h1} and \eqref{coj:h2}. This is presented in Section \ref{sec:upper}, where we also give a proof of the lower bounds in Theorems \ref{thm:H1Z}, \ref{thm:H23Z} and \ref{thm:uep}.

\section{Preliminaries} 
\label{sec:preliminary}
	We introduce some notation and concepts which are used throughout the article.
	The space of Schwartz functions is denoted by $\css(\RR^\ell)$. The Fourier transform of a function $g\in\css(\RR^\ell)$ is defined with the normalization
	\begin{equation*}
		\cff g(\xi)=\int_{\RR^\ell}e^{-i \xi\cdot x}g(x)dx\,,
	\end{equation*} 
	so that the inverse Fourier transform is given by $\cff^{-1}g(\xi)=(2 \pi)^{-\ell}\cff g(- \xi)$. The Plancherel identity with this normalization reads
	\begin{equation*}
		\int_{\RR^\ell}|f(x)|^2dx=\frac1{(2 \pi)^\ell}\int_{\RR^\ell}|\cff f(\xi)|^2 d \xi\,.
	\end{equation*}

	Let us now describe stochastic integrations with respect to $W$. We can interpret  $W$   as a Brownian motion with values in an infinite dimensional Hilbert space. In this context,  the stochastic integration theory with respect to $W$ can be handled by classical theories (see for example, \cite{DQ}).
		We briefly recall the main   features of this theory.

		We denote by  $\HH_0$ the Hilbert space defined as the closure of  $\mathcal{S}(\RR^\ell)$ under the inner product
		  \begin{equation}\label{HH0}
		  \langle g, h \rangle_{ \HH_0}=\frac 1{(2\pi)^\ell}  \int_{\RR^\ell}\mathcal{F}g(\xi)\overline{\mathcal{F}h(\xi)} \mu(\xi) d \xi\,.
		  \end{equation}
		which can also be written as 
		\begin{equation}\label{HH0.Cartest}
		  \langle g, h \rangle_{ \HH_0}=\iint_{\RR^\ell\times\RR^\ell}g(x)h(y)\gamma(x-y)dxdy \,.
		\end{equation}
		If $\gamma$ satisfies \ref{Cov:bounded}, then $\HH_0$ contains distributions such as Dirac delta masses. 
		The Gaussian family $W$ can be extended to an  {\it isonormal Gaussian} process $\{W(\phi), \phi \in L^2(\RR_+, \HH_0)\}$ parametrized by the Hilbert space $\HH:=L^2(\RR_+, \HH_0)$.
		For any $t\ge0$, let $\mathcal{F}_{t}$ be the $\sigma$-algebra generated by $W$ up to time $t$. 
		Let $\Lambda$ be the space of  $\HH_0$-valued predictable processes $g$  such that
		 $\EE\|g\|_{\HH}^{2}<\infty$. Then, one can construct (cf. \cite{HLN15}) the stochastic integral 
 $\int_0^\infty\int_{\mathbb{R}^\ell}g(s,x) \, W(ds,dx)$ 
  		such that
		\begin{equation}\label{int isometry}
		\EE  \lt( \int_0^\infty\int_{\mathbb{R}^\ell}g(s,x) \, W(ds,dx) \rt)^{2} 
		=
		\EE  \|g\|_{\HH}^{2}.
		\end{equation}
		Stochastic integration over finite time interval can be defined easily
		\begin{equation*}
		 	 \int_0^t\int_{\mathbb{R}^\ell}g(s,x) \, W(ds,dx)= \int_0^\infty\int_{\mathbb{R}^\ell}1_{[0,t]}(s) g(s,x) \, W(ds,dx)\,.
		\end{equation*} 
		Finally, the Burkholder's inequality in this context reads
		\begin{equation} \label{ineq:Burkholder}
			\lt\|\int_0^t\int_{\RR^\ell}g(s,x)W(ds,dx)\rt\|_{L^p(\Omega)}\le \sqrt{4 p}\lt\|\int_0^t\|g(s,\cdot)\|^2_{\HH_0}ds \rt\|^{\frac12}_{L^{\frac p2}(\Omega)}\,,
		\end{equation}
		which holds for all $p\ge2$ and $g\in \Lambda$. A useful application of \eqref{ineq:Burkholder} is the following result
		\begin{lemma}\label{lem:Wm}
		 Let $m\ge2$  be an integer, $f$ be a deterministic function on $[0,\infty)\times\RR^\ell$ and $u=\{u(s,x): s\ge0,x\in\RR^\ell\}$ be a predictable random field such that
		\begin{equation*}
			\U_m(s):=\sup_{x\in\RR^\ell}\|u(s,x)\|_{L^m(\Omega)}<\infty\,.
		\end{equation*}
		Under hypothesis \ref{Cov:Dalang}, we have
		\begin{equation*}
			\lt\|\int_0^t\int_{\RR^\ell}f(s,y)u(s,y)W(ds,dy) \rt\|_{L^m(\Omega)}
			\le \sqrt{4m}  \||f(s,y)| \mathbf{1}_{[0,t]}(s)\U_m(s)\|_{\HH_{s,y}}\,;
		\end{equation*}
		and under hypothesis \ref{Cov:bounded}, we have
		\begin{align*}
			\left\| \int_0^t \int_{\RR^{\ell}} f(s,y)u(s,y) W(ds,dy)\right\|_{L^m(\Omega)} \leq \sqrt{4m \gamma	(0)} 
			 \left( \int_0^t \left(\int_{\RR^{\ell}} f(s,y) dy \U_m(s)\right)^2 ds \right)^{\frac12}
		\end{align*}
	\end{lemma}
	\begin{proof}
		We consider only the hypothesis \ref{Cov:Dalang}, the other case is obtained similarly. In view of Burkholder inequality \eqref{ineq:Burkholder} and Minkowski inequality, it suffices to show
		\begin{equation}\label{tmp:fz}
			\int_0^t\| \|f(s,\cdot)u(s,\cdot)\|^2_{\HH_0}\|_{L^{\frac m2}(\Omega)} ds\le \||f(s,y)| \mathbf{1}_{[0,t]}(s)\U_m(s)\|_{\HH_{s,y}}^2\,.
		\end{equation}
		In fact, using \eqref{HH0.Cartest} and Minkowski inequality, the left-hand side in the above is at most
		\begin{equation*}
			\int_0^t\iint_{\RR^\ell\times \RR^\ell}|f(s,x)f(s,y)| \|u(s,x)u(s,y)\|_{L^{\frac m2}(\Omega)}\gamma(x-y)dxdyds\,.
		\end{equation*}
		Note in addition that by Cauchy-Schwarz inequality,
		\[ \|u(s,x)u(s,y)\|_{L^{\frac m2}(\Omega)}\le \|u(s,x)\|_{L^m(\Omega)}^{1/2}\|u(s,y)\|_{L^m(\Omega)}^{1/2}\le\U_m(s). \]
		From here, \eqref{tmp:fz} is transparent and the proof is complete.
	\end{proof}

		We now state the definition of the solution to equation  (\ref{eqn:SHE}) using the stochastic integral introduced previously.
		\begin{definition}\label{def-sol-sigma}
			Let $u=\{u(t,x),  t\ge 0, x \in \mathbb{R}^\ell\}$ be a real-valued predictable stochastic process  such that for all $t \ge 0$ and $x\in\RR^\ell$ the process $\{p_{t-s}(x-y)u(s,y) \mathbf{1}_{[0,t]}(s), 0 \leq s \leq t, y \in \mathbb{R}^\ell\}$ is an element of $\Lambda$.
			We say that $u$ is a mild solution of  (\ref{eqn:SHE}) if for all $t \in [0,T]$ and $x\in \mathbb{R}^\ell$ we have
			\begin{equation}\label{eq:mild-formulation sigma}
				u(t,x)=p_t*u_0 (x) + \int_0^t \int_{\mathbb{R}^\ell}p_{t-s}(x-y)u(s,y) W(ds,dy) \quad a.s.
			\end{equation}
		\end{definition}
		The following existence and uniqueness result has been proved in \cite{HLN15}{ under hypothesis \ref{Cov:Dalang}. Under hypothesis \ref{Cov:bounded}, one can proceed as in \cite{Huang}, using a simple Picard iteration argument to obtain the existence and uniqueness of the solution.}
		\begin{theorem}  \label{thm1}
		Suppose that $u_0$ satisfies \eqref{con:u0} and the spectral  measure $\mu$ satisfies  hypotheses \ref{Cov:bounded} or \ref{Cov:Dalang}.    Then there exists a unique solution to equation (\ref{eqn:SHE}).
		\end{theorem}
	When $u_0=\delta(\cdot-z)$, we denote the corresponding unique solution by $\Z(z; t, x)$. 
	In particular $\Z(z;\cdot,\cdot) $ is predictable and satisfies
	\begin{equation}\label{eqn:Z}
		\Z (z;t,x)=p_t(x-z)+\int_0^t\int_{\RR^\ell} p_{t-s}(x-y)\Z(z;s,y)W(ds,dy)
	 \end{equation}
	for all $t\ge0$ and $x\in\RR^\ell$. 	

	Next, we record a Gronwall-type lemma which will be useful later.
	\begin{lemma}\label{lem:gronwall}
		Suppose $\alpha\in[0,2)$ and $f$ is a locally bounded function on $[0,\infty)$ such that
		\begin{equation*}
			f_t\le A\int_0^t\lt(\frac{s(t-s)}t \rt)^{-\frac \alpha2}f_s ds+Bg_t \quad\mbox{for all}\quad t\ge0 \,,
		\end{equation*}
		where $A,B$ are positive constants and $g$ is non-decreasing function. Then there exists a constant $C_ \alpha$ such that
		\begin{equation*}
			f_t\le 2Bg_t e^{C_ \alpha A^{\frac2{2- \alpha}}t} \quad\mbox{for all}\quad t\ge0\,.
		\end{equation*}
	\end{lemma}
	\begin{proof}
		Fix $T>0$. For each $\rho>0$, denote $D_ \rho=\sup_{t\in[0,T]}f_t e^{-\rho t}$. It follows that
		\begin{align*}
			D_ \rho\le A\int_0^t\lt(\frac{s(t-s)}t \rt)^{-\frac \alpha2} e^{-\rho(t-s)}ds D_ \rho+Bg_T\,.
		\end{align*}
		It is easy to see
		\begin{align*}
			\int_0^t\lt(\frac{s(t-s)}t \rt)^{-\frac \alpha2} e^{-\rho(t-s)}ds
			&\le 2\int_{\frac t2}^t\lt(\frac{s(t-s)}t \rt)^{-\frac \alpha2} e^{-\rho (t-s)}ds
			\\&\le 2^{1+\frac \alpha2}\int_0^\infty s^{-\frac \alpha2}e^{-\rho s}ds
			\\&\le C \rho^{-\frac {2-\alpha}2}
		\end{align*}
		for some suitable constant $C$ depending only on $\alpha$. We then choose $\rho=(2AC)^{\frac2{2- \alpha}} $ so that $AC\rho^{-\frac{2- \alpha}2}=\frac12$. This leads to $D_ \rho\le 2Bg_T$, which implies the result.
	\end{proof}
	Let us conclude this section by introducing a few key notation which we will use throughout the article.
	Let $B=(B(t),t\ge0)$ denote a standard Brownian motion in $\RR^\ell$ starting at the origin. For each $t>0$, we denote
	\begin{equation}\label{def:Bbridge}
		B_{0,t}(s)=B(s)-\frac st B(t)\quad\forall s\in[0,t]\,.
	\end{equation}
	The process $B_{0,t}=(B_{0,t}(s),0\le s\le t)$ is independent from $B(t)$ and is a Brownian bridge which starts and ends at the origin.
	An important connection between $B$ and $B_{0,t}$ is the following identity. For every $\lambda\in(0,1)$ and every bounded measurable function $F$ on $C([0,\lambda t];\RR^d) $ we have
	\begin{multline}\label{id.BBBM}
		\EE\lt[F(\{B_{0,t} (s);0\le s\le \lambda t\}) \rt]
		\\=(1- \lambda)^{-\frac d2}\EE\lt[\exp\lt\{-\frac{|B(\lambda t)|^2}{2(1- \lambda)t} \rt\}F(\{B (s);0\le s\le \lambda t\})  \rt]\,.
	\end{multline}
	This is in fact an application of Girsanov's theorem, see \cite{HLN15}*{Eq. (2.8)} for more details. 
	Let $B^1,B^2,\dots$ be independent copies of $B$ and $B^{1}_{0,t},B^2_{0,t},\dots$ be the corresponding Brownian bridges. An important quantity which appears frequently in our consideration is
	\begin{equation}\label{def:Thetam}
		\Theta_t(m):=\sup_{s\in(0,t]} \EE\exp\lt\{\int_0^s\sum_{1\le j<k\le m}\gamma(B_{0,s}^j(r)-B_{0,s}^k(r))dr \rt\}\,.
	\end{equation}  	   
	From the proof of Proposition 4.2 in \cite{HLN15}, it is easy to see that under one of the hypotheses \ref{Cov:bounded} and \ref{Cov:Dalang}, $\Theta_t(m) < \infty$ for any $t>0$. Finally, $A\lesssim E$ means $A\le CE$ for some positive constant $C$, independent from all the terms appearing in $E$.
\section{Variations}\label{sec:variations}
		We introduce two variational quantities and give their basic properties and relations. The high moment asymptotic is governed by a variational quantity which  is known as the Hartree energy (cf. \cite{ChPh15}). 
		If there exists a locally integrable function $\gamma$ whose Fourier transform is $\mu$, then the Hartree energy can be expressed as
		\begin{equation}\label{eqn:ceeg}
			\cee_H(\gamma)=\sup_{g\in\cgg}\lt\{\int_{\RR^\ell}\int_{\RR^\ell}\gamma(x-y)g^2(x)g^2(y)dxdy-\int_{\RR^\ell}|\nabla g(x) |^2dx \rt\}\,,
		\end{equation}
		where $\cgg$ is the set
		\begin{equation}\label{eq:G}
			\cgg=\lt\{g\in W^{1,2}(\RR^\ell):\|g\|_{L^2(\RR^\ell)}=1\rt\}\,.
		\end{equation}
		The subscript $H$ stands for ``Hartree''.
		We can also write this variation in Fourier mode. Indeed, the presentation \eqref{eqn:gamspec} leads to
		\begin{align*}
			\iint_{\RR^\ell\times\RR^\ell}\gamma(x-y)g^2(x)g^2(y)dxdy
			&=(2 \pi)^{-\ell}\int_{\RR^\ell}|\cff[g^2](\xi)|^2\mu(\xi)d \xi
			\\&=(2 \pi)^{-3\ell}\int_{\RR^\ell}|\cff g*\cff g(\xi)|^2\mu(\xi)d \xi\,.
		\end{align*}
		Setting $h=(2 \pi)^{-\frac\ell2}\cff g$ so that $\|h\|_{L^2}=1$, we arrive at
		\begin{equation}\label{eqn:cee}
			\cee_H(\gamma)=\sup_{h\in\caa}\lt\{(2 \pi)^{-\ell} \int_{\RR^\ell} |h*h(\xi)|^2 \mu( \xi)d \xi-\int_{\RR^\ell}|h(\xi)|^2|\xi|^2 d \xi \rt\}
		\end{equation}
		where
		\begin{equation*}
			\caa=\lt\{h:\RR^\ell\to\CC\,\Big|\,\|h\|_{L^2(\RR^\ell)}=1,\int_{\RR^\ell}|\xi|^2|h(\xi)|^2 d \xi<\infty \mbox{ and } \overline{h(\xi)}=h(- \xi)\rt\}\,.
		\end{equation*}
		Under \ref{Cov:bounded}, from \eqref{eqn:ceeg}, we upper bound $\gamma(x-y)$ by $\gamma(0)$, it follows that $\cee_H(\gamma)\le \gamma(0)$, which is finite.
		The fact that this variation (either in the form \eqref{eqn:ceeg} or \eqref{eqn:cee}) is finite under the condition \ref{Cov:Dalang} is not immediate. In some special cases, this is verified in \cite{MR3414457} and  \cite{Chetal16}.
		\begin{proposition}
			Suppose \eqref{k5} holds. 
			Then $\cee_H(\gamma)$ is finite.
		\end{proposition}
		\begin{proof}
			Our proof is based on the argument in \cite{Chetal16}*{Proposition 3.1}. Here, however, we work on the frequency space and use the presentation \eqref{eqn:cee}. Let $h$ be in $\caa$. Applying Cauchy-Schwarz inequality yields
			\begin{align*}
				|h*h(\xi)|^2=\lt|\int_{\RR^\ell}h(\xi- \xi')h(\xi')d \xi'\rt|^2
				\le \int_{\RR^\ell}|h(\xi- \xi')|^2d \xi'\int_{\RR^\ell}|h(\xi')|^2d \xi'=1\,.
			\end{align*}
			On the other hand, using the elementary inequality
			\begin{equation*}
				|\xi|^2\le 2|\xi- \xi'|^2+2|\xi'|^2
			\end{equation*}
			and Cauchy-Schwarz inequality, we also get
			\begin{align*}
				|\xi|^2 |h*h(\xi)|^2
				&\le {2}\lt|\int_{\RR^\ell}h(\xi -\xi')|\xi- \xi'|h(\xi')d \xi' \rt|^2+{2}\lt|\int_{\RR^\ell}h(\xi -\xi')|\xi'|h(\xi')d \xi' \rt|^2
				\\&\le{ 4}\int_{\RR^\ell}|h(\xi')|^2|\xi'|^2 d \xi'\,.
			\end{align*}
			Then, for every $R>0$ we have
			\begin{align*}
				\int_{\RR^\ell}|h*h(\xi)|^2 \mu(\xi)d \xi
				&=\int_{|\xi|\le R}|h*h(\xi)|^2 \mu(\xi)d \xi+\int_{|\xi|>R}|h*h(\xi)|^2 \mu(\xi)d \xi
				\\&\le \int_{|\xi|\le R} \mu(\xi)d \xi+{4}\int_{|\xi|>R}\frac{ \mu(\xi)}{|\xi|^2} d \xi\int_{\RR^\ell}|h(\xi)|^2|\xi|^2 d \xi\,.
			\end{align*}
			We now choose $R$ sufficiently large so that ${ 4}(2 \pi)^{-\ell} \int_{|\xi|>R}\frac{ \mu(\xi)}{|\xi|^2} d \xi<1$. This implies
			\begin{align*}
				(2 \pi)^{-\ell}\int_{\RR^\ell}|h*h(\xi)|^2 \mu(\xi)d \xi-\int_{\RR^\ell}|h(\xi)|^2|\xi|^2 d \xi \le (2 \pi)^{-\ell}\int_{|\xi|\le R} \mu(\xi)d \xi
			\end{align*}
			for all $g$ in $\caa$, which finishes the proof.
		\end{proof}
		In establishing the lower bound of spatial asymptotic, another variation arises, which is given by
		\begin{equation}\label{eqn:rootcee}
			\cmm(\gamma)=\sup_{g\in\cgg}\lt\{\lt(\int_{\RR^\ell}\int_{\RR^\ell}\gamma(x-y)g^2(x)g^2(y)dxdy\rt)^{\frac12}-\frac12\int_{\RR^\ell}|\nabla g(x) |^2dx \rt\}\,,
		\end{equation}
		or alternatively in frequency mode
		\begin{equation}\label{eqn:rootceefourier}
			\cmm(\gamma)=\sup_{h\in\caa}\lt\{\lt((2 \pi)^{-\ell} \int_{\RR^\ell} |h*h(\xi)|^2 \mu( \xi)d \xi\rt)^{\frac12}-\frac12\int_{\RR^\ell}|h(\xi)|^2|\xi|^2 d \xi \rt\}\,.
		\end{equation}
		Under the scaling condition \ref{con:scaling}, $\cee_H$ and $\cmm$ are linked together by the following result. 
		\begin{proposition}\label{prop:ME} Assuming condition \ref{con:scaling}, $\cee_H(\gamma)$ is finite if and only if $\cmm(\gamma)$ is finite. In addition,
			\begin{equation*}
				\cmm(\gamma)=\frac{4- \alpha}4 \lt( \frac{2\cee_H(\gamma)}{2- \alpha} \rt)^{\frac{2- \alpha}{4- \alpha}}\,.
			\end{equation*}
		\end{proposition}
		Before giving the proof, let us see how \eqref{eqn:ceeg} and \eqref{eqn:rootcee} are connected to a certain interpolation inequality.
		Under scaling condition \ref{con:scaling}, it is a routine procedure in analysis to connect the finiteness of $\cee_H(\gamma)$ with a certain interpolation inequality. For instance, when $\gamma= \delta$ and $\ell=1$, the fact that
		\begin{equation*}
			\sup_{g\in\cgg}\lt\{\int_{\RR}g^4(x)dx-\int_\RR|g'(x)|^2dx \rt\}<\infty
		\end{equation*}
		is equivalent to the following Gagliardo--Nirenberg inequality
		\begin{equation*}
			\|g\|_{L^4}\le C\|g\|^{3/4}_{L^2}\|g'\|^{1/4}_{L^2}
		\end{equation*}
		for all $g$ in $W^{1,2}(\RR)$. For readers convenience, we provide a brief explanation below. 
		\begin{proposition}\label{prop:interpolation}
			Assume that the scaling relation \ref{con:scaling} holds. 
			
			\noindent{\upshape(i)} If $\cee_H(\gamma)$ is finite then there exists $\kappa>0$ such that for all $g$ in $W^{1,2}(\RR^\ell)$
			\begin{equation}\label{ineq:interpolation}
				\iint_{\RR^\ell\times\RR^\ell}\gamma(x-y)g^2(x)g^2(y)dxdy
				\le  \kappa \lt(\int_{\RR^\ell}|g(x) |^2dx \rt)^{2-\frac \alpha2} 
				\lt(\int_{\RR^\ell}|\nabla g(x) |^2dx \rt)^{\frac \alpha2} \,,
			\end{equation}
			In addition the constant $\kappa$ can be chosen to be
			\begin{equation}\label{id:bestkappa}
			 	\kappa:= \kappa(\gamma) :=\frac2 \alpha \lt(\frac{\alpha }{2- \alpha}\cee_H(\gamma) \rt)^{\frac{2- \alpha}{2}}\,.
			\end{equation}

			\noindent {\upshape(ii)} If \eqref{ineq:interpolation} holds for some finite constant $\kappa>0$, then $\cee_H(\gamma)$ is finite and the best constant in \eqref{ineq:interpolation} is $\kappa(\gamma)$.
		\end{proposition}
		\begin{proof}
		Recall that $\cgg$ is defined in \eqref{eq:G}. 
			
			(i) Let $g$ be in $\cgg$. For each $\theta>0$, the function $x\mapsto g_ \theta(x):=\theta^{\frac\ell2}{g(\theta x)}$ also belongs to $\cgg$. Hence,
			\begin{align*}
				\iint_{\RR^\ell\times\RR^\ell}\gamma(x-y)g_ \theta^2(x)g_ \theta^2(y)dxdy-\int_{\RR^\ell}|\nabla g_ \theta(x)|^2 dx\le \cee_H(\gamma)\,.
			\end{align*}
			Writing these integrals back to $g$ and using \ref{con:scaling} yields
			\begin{align*}
				\theta^ \alpha\iint_{\RR^\ell\times\RR^\ell}\gamma(x-y)g^2(x)g^2(y)dxdy- \theta^2\int_{\RR^\ell}|\nabla g(x)|^2 dx\le \cee_H(\gamma)
			\end{align*}
			for all $\theta>0$. Optimizing the left-hand side (with respect to $\theta$) leads to
			\begin{align*}
				\frac{2- \alpha}{\alpha}\lt(\frac\alpha2\rt)^{\frac2{2- \alpha}}\lt(\iint_{\RR^\ell\times\RR^\ell}\gamma(x-y)g^2(x)g^2(y)dxdy \rt)^{\frac2{2- \alpha}}\le \cee_H(\gamma)\lt(\int_{\RR^\ell}|\nabla g(x)|^2 dx\rt)^{\frac \alpha{2- \alpha}}\,.
			\end{align*}
			Removing the normalization $\|g\|_{L^2}=1$ and some algebraic manipulation yields the result.

			\noindent (ii) Let $\kappa_0$ be the best constant in \eqref{ineq:interpolation}. Then for every $g\in\cgg$, 
			\begin{align*}
				\iint_{\RR^\ell\times\RR^\ell}\gamma(x-y)g^2(x)g^2(y)dxdy-\int_{\RR^\ell}|\nabla g(x)|^2 dx
				&\le \kappa_0 \|\nabla g\|^{\alpha}_{L^2}-\|\nabla g\|^{2}_{L^2}
				\\&\le \sup_{\theta>0}\{\kappa_0 \theta^ \alpha- \theta^2 \}=\frac{2- \alpha}{\alpha}( \frac\alpha2 \kappa_0)^{\frac 2{2- \alpha}}\,.
			\end{align*}
			This shows $\cee_H(\gamma)$ is finite and at most $\frac{2- \alpha}{\alpha}( \frac\alpha2 \kappa_0)^{\frac 2{2- \alpha}}$, which also means $\kappa(\gamma)\le \kappa_0$. On the other hand, (i) already implies $\kappa_0\le \kappa(\gamma)$, hence completes the proof. 
		\end{proof}
		\begin{proof}[Proof of Proposition \ref{prop:ME}]
			Reasoning as in Proposition \ref{prop:interpolation}, we see that $\cmm(\gamma)$ is finite if and only if \eqref{ineq:interpolation} holds for some constant $\kappa>0$. In addition, the best constant $\kappa(\gamma)$ in \eqref{ineq:interpolation} satisfies the relation
			\begin{equation*}
				\cmm(\gamma)=\frac{4- \alpha}{4} \lt(\frac \alpha2\rt)^{\frac \alpha{4- \alpha}}(\kappa(\gamma))^{\frac2{4- \alpha}}\,.
			\end{equation*}
			Together with \eqref{id:bestkappa}, this yields the result.
		\end{proof}
		{The following result preludes the connection between $\cee_H,\cmm$ with exponential functional of Brownian motions}.
		\begin{lemma}\label{lem:expBB}
		Let $\{B(s),s\ge0\}$ be a Brownian motion in $\RR^n$ and $D$ be a bounded open domain in $\RR^n$ containing 0. Let $h(s,x)$ be a bounded function defined on $[0,1]\times\RR^n$ which is continuous in $x$ and equicontinuous (over $x\in\RR^n$) in $s$. Then
		\begin{multline}\label{lim:fBB}
			\lim_{t\to\infty}\frac1t\log \EE \lt[\exp\left\{\int_0^t h\left(\frac st,B(s)-\frac st B(t)\right)ds \right\};\tau_D\ge t\rt]
			\\=\int_0^1\sup_{g\in \cgg_D }\left\{\int_{D} h(s,x) g^2(x)dx-\frac12\int_{D}|\nabla g(x)|^2dx \right\}ds\,,
		\end{multline}
		where $\cgg_D$ is the class of functions $g$ in $W^{1,2}(\RR^n)$ such that $\int_{D}|g(x)|^2dx=1$ and $\tau_D$ is the exit time $\tau_D := \inf\{t\geq 0: B_t \notin D\}$. 
	\end{lemma}
	\begin{proof}
		The process $\{B_{0,t}(s)=B(s)-\frac st B(t)\}_{s\in[0,t]}$ is a Brownian bridge.
		We fix $\theta\in(0,1)$ and consider first the limit
		\begin{equation*}
			\lim_{t\to\infty}\frac1t\log \EE \lt[\exp\left\{\int_0^{\theta t} h\left(\frac st,B_{0,t}(s)\right)ds \right\};\tau_D\ge t\rt]\,.
		\end{equation*}
		Let $M$ be such that $|x|\le M$ for all $x\in D$. Using Girsanov theorem (see \cite{HLN15}*{Eq. (2.38)}), we can write
		\begin{align*}
			&\EE \lt[\exp\left\{\int_0^{\theta t} h\left(\frac st,B_{0,t}(s)\right)ds \right\};\tau_D\ge t\rt]
			\\&=(1- \theta)^{-\frac n2}\EE \lt[\exp\left\{\int_0^{\theta t} h\left(\frac st,B(s)\right)ds-\frac{|B(\theta t)|^2}{2t(1- \theta)} \right\};\tau_D\ge t\rt]
			\\&\ge(1- \theta)^{-\frac n2}\EE \lt[\exp\left\{\int_0^{\theta t} h\left(\frac st,B(s)\right)ds-\frac{M^2}{2t(1- \theta)} \right\};\tau_D\ge t\rt]\,.
		\end{align*}
		The result of \cite[Proposition 3.1] {MR3414457} asserts that
		\begin{multline*}
			\lim_{t\to\infty}\frac1t\log \EE \lt[\exp\left\{\int_0^{\theta t} h\left(\frac st,B(s)\right)ds \right\};\tau_D\ge t\rt]
			\\=\int_0^ \theta\sup_{g\in\cgg_D}\left\{\int_{D} h(s,x) g^2(x)dx-\frac12\int_{D}|\nabla g(x)|^2dx \right\}ds
		\end{multline*}
		This leads to
		\begin{multline}\label{tmp:infhg}
			\liminf_{t\to\infty}\frac1t\log \EE \lt[\exp\left\{\int_0^{\theta t} h\left(\frac st,B_{0,t}(s)\right)ds \right\};\tau_D\ge t\rt]
			\\\ge \int_0^ \theta\sup_{g\in\cgg_D}\left\{\int_{D} h(s,x) g^2(x)dx-\frac12\int_{D}|\nabla g(x)|^2dx \right\}ds\,.
		\end{multline}
		Observing that
		\begin{equation*}
			\lt|\log\EE\exp\left\{\int_0^t h\left(\frac st,B_{0,t}(s)\right)ds \right\}-\log\EE \exp\left\{\int_0^{\theta t} h\left(\frac st,B_{0,t}(s)\right)ds \right\}\rt|\le (1- \theta)t\|h\|_\infty\,,
		\end{equation*}
		we can send $\theta\uparrow1$ in \eqref{tmp:infhg} to obtain the lower bound for \eqref{lim:fBB}.
		The upper bound for \eqref{lim:fBB} is proved analogously, we omit the details.
	\end{proof}
		We conclude this section with an observation: \ref{con:scaling} induces the following scaling relation on $\cee_H(\gamma)$
		\begin{equation}\label{scal:cee}
			\cee_H(\lambda \gamma)=\lambda^{\frac{2}{2- \alpha}}\cee_H(\gamma)\quad \mbox{for all}\quad \lambda>0\,.
		\end{equation}
\section{Feynman-Kac formulas and functionals of Brownian Bridges} 
\label{sub:feynman_kac_formula}
	We derive Feynman-Kac formulas for the moments $\EE u^m(t,x)$ for integers $m\ge2$. These formulas play important roles in proving upper and lower bounds of \eqref{coj:h1} and \eqref{lim:H2ep}.
	
	To discuss our contributions in the current section, let us assume for the moment that $\dot{W}$ is a space-time white noise and $\ell=1$. The most well-known Feynman-Kac formula for second moment is
	\begin{equation*}
		\EE [(u(t,x))^2]=\EE \left(\prod_{j=1}^2 u_0(B^j(t)+x)\exp\lt\{\int_0^t \delta(B^1(s)-B^2(s))ds \rt\}\right)\,,
	\end{equation*}
	where $B^1,B^2$ are two independent Brownian motions starting at 0.
	If $u_0$ is merely a measure, some efforts are needed to make sense of $u_0(B(t)+x)$, which appears on the right-hand side above. An attempt is carried out in \cite{ChenNualart2016} using Meyer-Watanabe's theory of Wiener distributions.

	The Feynman-Kac formulas presented here (see \eqref{eqn:FKbridge} below) have appeared in \cite{HLN15}. However, there seems to have a minor gap in that article. Namely, Eq. (4.52) there has not been proven if $u_0$ is a measure. In the current article, we take the chance to fill this gap. Our approach is in the same spirit as \cite{HLN15} and is different from \cite{ChenNualart2016}. In particular, we do not make use of Wiener distributions.

	Since $W_\epsilon$ has bounded covariance, it is easy to see that the stochastic heat equation
	\begin{equation}\label{eqn:SHEpsilon}
		u_ \epsilon(t,x)=p_t*u_0(x)+\int_0^t\int_{\RR^\ell}p_{t-s}(x-y)u_ \epsilon(s,y)W_ \epsilon(ds,dy)
	\end{equation}
	has a unique random field solution $u_ \epsilon$. In addition, for each $t>0$ and $x\in\RR^\ell$, $u_ \epsilon(t,x)$ admits a chaos expansion (see, for instance \cite{HuNu09})
	\begin{equation}\label{eqn:uep:chaos}
		u_ \epsilon(t,x)=\sum_{n=0}^\infty I_{\epsilon,n}(f_n[u_0](t,x;\cdot))
	\end{equation}
	where $f_0[u_0](t,x)=p_t*u_0(x)$ and for each $n\ge1$
	\begin{multline}\label{def:fn}
		f_n[u_0](t,x;s_1,x_1,\dots,s_n,x_n)
		\\=\frac1{n!}p_{t-s_{\sigma(n)}}(x- x_{\sigma(n)})\cdots p_{s_{\sigma(2)}-s_{\sigma(1)}}(x_{\sigma(2)}-x_{\sigma(1)})p_{s_{\sigma(1)}}*u_0(x_{\sigma(1)}) \,.
	\end{multline}
	Here, $\sigma$ denotes the permutation of $\{1,2\dots,n\}$ such that $0<s_{\sigma(1)}<\cdots<s_{\sigma(n)}<t$ and $I_{\epsilon,n}$ is the $n$-th multiple It\^o-Wiener integral with respect to the Gaussian field $W_ \epsilon$. 

	\begin{proposition}\label{prop:repu}
		Let $u_0$ be a measure satisfying \eqref{con:u0}. Then
		\begin{equation}\label{rep:u}
			u(t,x)=\int_{\RR^\ell}\Z(z;t,x)u_0(dz)\,.
		\end{equation}
		In addition, if \ref{Cov:bounded} holds, then
		\begin{align}\label{eqn:FKZ}
			\frac{\Z(z; t,x)}{p_t(z-x)} 
			= \EE_{B} \exp \left\{ \int_0^t \int_{\RR^\ell} \delta \left( B_{0,t}(t-s)+ \frac{t-s}{t}z + \frac{s}{t} x -y \right) W(ds,dy) - \frac{t}{2}\gamma(0)\right\}\,.
		\end{align}
	\end{proposition}
	\begin{proof} 
	Let $v(t,x)$ be the integral on the right-hand side of \eqref{rep:u}. From \eqref{eqn:Z}, integrating $z$ with respect to $u_0(dz)$ and applying the stochastic Fubini theorem (cf. \cite[Theorem 4.33]{DPZ}), we have
	\begin{align*}
		v(t,x)
		&=\int_{\RR^\ell}p_t(x-z)u_0(dz) +\int_{\RR^\ell}\int_0^t\int_{\RR^\ell}p_{t-s}(x-y)\Z(z;s,y)W(ds,dy)u_0(dz)
		\\&=p_t*u_0(x) +\int_0^t\int_{\RR^\ell}p_{t-s}(x-y)v(s,y)W(ds,dy)\,.
	\end{align*}
	Hence, $v$ is a solution of \eqref{eqn:SHE} with initial datum $u_0$. By unicity, Theorem \ref{thm1}, we see that $u=v$ and \eqref{rep:u} follows.

	Next, we show \eqref{eqn:FKZ} assuming \ref{Cov:bounded}. Fix $t>0$ and $x\in\RR^\ell$. For every $u_0\in C_c^\infty(\RR^\ell)$, the following Feynman-Kac formula holds
		\begin{equation*}
			u(t,x)=\EE_B u_0(B(t)+x)\exp\lt\{\int_0^t\int_{\RR^\ell}\delta(B(t-s)+x-y)W(ds,dy)-\frac t2 \gamma(0) \rt\}
		\end{equation*}
		Using the decomposition \eqref{def:Bbridge}	and the fact that $B_{0,t}$ and $B(t)$ are independent, we see that
		\begin{align}\label{tmp:uY}
		 	u(t,x)
		 	=\int_{\RR^\ell}Y(z;t,x) p_t(z) u_0(z+x) dz
		\end{align} 
		where
		\begin{align*}
			Y(z;t,x)&=\EE_B \exp\lt\{\int_0^t\int_{\RR^\ell}\delta(B_{0,t}(t-s)+\frac{t-s}t z+x-y)W(ds,dy)-\frac t2 \gamma(0) \rt\}\\
			& =\EE_B \exp\lt\{ V_{t,x}(z)\rt\}\,.
		\end{align*}
		Together with \eqref{rep:u} we obtain
		\begin{equation*}
			\int_{\RR^\ell}\Z(z;t,x)u_0(z)dz=\int_{\RR^\ell}Y(z-x;t,x) p_t(z-x) u_0(z) dz
		\end{equation*}
		for all $u_0\in C^\infty_c(\RR^\ell)$.

		Next we show that $z\mapsto Y(z;t,x)$ is continuous. Fix $p>2$. From the elementary relation $|e^x-e^y| \leq (e^x + e^y)|x-y|$ and the Cauchy-Schwarz inequality, it follows 
		\begin{align*}
		&\EE \left| Y(z; t,x) - Y(z'; t,x) \right|^p\\
		 &\leq \left(\EE_W \left( \EE_B\left[e^{V_{t,x}(z)}+e^{V_{t,x}(z')}\right]^2 \right)^p\right)^{1/2} \left(\EE_W \left(\EE_B|V_{t,x}(z)-V_{t,x}(z')|^2\right)^p\right)^{1/2}\,.
		\end{align*}
		Since $\gamma$ is bounded, it is easy to see that
		\begin{equation*}
			\sup_{z,x\in\RR^\ell}\EE e^{2pV_{t,x}(z)}\le C_{p,t}
		\end{equation*}
		for some constant $C_{p,t}$.
		We now resort to Minkowski inequality, our exponential bound for $V_{t,x}(z)$ and the relation between $L^p$ and $L^2$ moments for Gaussian random variables in order to obtain 
		\begin{equation*}
			\EE \left| Y(z; t,x) - Y(z'; t,x) \right|^p \leq C_{p,t} \left( \EE|V_{t,x}(z)-V_{t,x}(z')|^2 \right)^{p/2}	\,.
		\end{equation*}
		In addition, under \ref{Cov:bounded}, $\gamma$ is H\"older continuous with order $\kappa >0$ at 0, it follows that
		 \begin{align*}
		 &\EE |V_{t,x}(z)-V_{t,x}(z')|^2\\
		 &= \EE \bigg(\int_0^t\int_{\RR^\ell}\delta(B_{0,t}(t-s)+\frac{t-s}t z+x-y)W(ds,dy)\\
		 &\quad \quad - \int_0^t\int_{\RR^\ell}\delta(B_{0,t}(t-s)+\frac{t-s}t z'+x-y)W(ds,dy) \bigg)^2\\
		 &= \int_0^t \left( \gamma(0) - \gamma\Big( \frac{t-s}{t} (z-z')\Big) \right)ds \lesssim t |z-z'|^{\kappa}\,.
		 \end{align*}	
		We have shown
		\begin{equation*}
			\EE \left| Y(z; t,x) - Y(z'; t,x) \right|^p \lesssim |z-z'|^{p \kappa}	\,.
		\end{equation*}
		Thus, the process $z\rightarrow Y(z;t,x)$  has a continuous version. On the other hand, $z\mapsto\Z(z;t,x)$ is also continuous (see Proposition \ref{prop:hder} below). It follows that
		$\Z(z;t,x)=Y(z-x;t,x) p_t(z-x)$, which is exactly \eqref{eqn:FKZ}.
	\end{proof}
	\begin{proposition}
		Assuming \ref{Cov:bounded}, we have
		\begin{multline}\label{h1:ZmFK}
			\EE\left[\prod_{j=1}^m \frac{\Z(z_j;t,x_j)}{p_t(x_j-z_j)} \right]
			\\=  
			\EE\exp\left\{\int_0^t\sum_{1\le j<k\le m} \gamma\left(B_{0,t}^j(s)-B_{0,t}^k(s)+\frac st(z_j-z_k) +\frac{t-s}t(x_j-x_k)\right)ds\right\}\,.
		\end{multline}
		and
		\begin{equation}\label{h1:supZm}
			\EE\left[\prod_{j=1}^m \frac{\Z(z_j;t,x_j)}{p_t(x_j-z_j)} \right]
			\le  
			\EE\exp\left\{\int_0^t\sum_{1\le j<k\le m} \gamma\left(B_{0,t}^j(s)-B_{0,t}^k(s)\right)ds\right\}\,.
		\end{equation}
	\end{proposition}
	\begin{proof}
		We observe that conditioned on $B$, 
		\begin{equation*}
		 	V(z,x):= \int_0^t\int_{\RR^\ell}\delta\left(B_{0,t}(t-s)+\frac{t-s}tz+\frac st x-y\right)W(ds,dy)
		\end{equation*} 
		is a normal random variable with mean zero. In addition, for every $x,x',z,z'\in\RR^\ell$, applying \eqref{Wep.cov}, we have
		\begin{multline}\label{tmp:covV}
			\EE\lt[V(B^j,z,x)V(B^k,z',x') \Big|B^j,B^k\rt] 
			\\= \int_0^t \gamma\lt(B_{0,t}^j(s)-B_{0,t}^k(s)+\frac st(z-z')+\frac{t-s}t(x-x') \rt)ds\,.
		\end{multline}  
		For every $(x_1,\dots,x_m)\in(\RR^\ell)^m$, using \eqref{eqn:FKZ} and \eqref{tmp:covV}, we have
		\begin{multline}\label{uep:mFK}
			\EE\left[\prod_{j=1}^m \frac{\Z(z_j;t,x_j)}{p_t(x_j-z_j)} \right]
			\\=  
			\EE\exp\left\{\int_0^t\sum_{1\le j<k\le m} \gamma\left(B_{0,t}^j(s)-B_{0,t}^k(s)+\frac st(z_j-z_k) +\frac{t-s}t(x_j-x_k)\right)ds\right\}
		\end{multline}
		Note that in the exponent above, the diagonal terms (with $j=k$) are removed because there are cancellations with the normalization factor $-\frac t2 \gamma(0)$ in \eqref{eqn:FKZ}, which occur after taking expectation with respect to $W$.
		Finally, apply \cite[Lemma 4.1] {HLN15}, we obtain \eqref{h1:supZm} from \eqref{uep:mFK}.
	\end{proof}
	To extend the previous result to nosies satisfying \ref{Cov:Dalang}, we need the following result.
	\begin{proposition}\label{prop:ZZep} Assuming \ref{Cov:Dalang}. There exists a constant $c$ depending only on $\alpha$ such that for any $\beta \in (0, 4\wedge (\ell -\alpha))$, 
		\begin{equation}
		\left\|\frac{\Z_{\epsilon}(x_0;t,x)}{p_t(x-x_0)} - \frac{\Z(x_0;t,x)}{p_t(x-x_0)} \right\|_{L^m(\Omega)}\lesssim  \epsilon^{\frac \beta4}t^{\frac{2- \alpha}4}\sqrt m \Theta_t^{\frac1m}(m)e^{cm^{\frac2{2- \alpha}}t}\quad\mbox{for all}\quad t\ge0
		\end{equation}
	where $\Theta_t(m)$ is defined in \eqref{def:Thetam}	
	\end{proposition}
	\begin{proof}
		Let us put
		\begin{equation*}
			M_s=\sup_{y\in\RR^\ell}\frac{\|\Z(x_0;s,y)-\Z_ \epsilon(x_0;s,y)\|_{L^m(\Omega)}}{p_t(y-x_0)}\,.
		\end{equation*}
		From \eqref{eqn:Z}, we have 
		\begin{align*}
		\frac{\Z(x_0;t,x)}{p_t(x-x_0)} &= 1 + \int_0^t \int_{\RR^\ell} \frac{p_{t-s}(x-y) p_s(y-x_0)}{p_t(x-x_0)} \frac{\Z(x_0;s,y)}{p_s(y-x_0)} W(ds,dy)\\
		&= 1+ \int_0^t \int_{\RR^\ell} p_{\frac{s(t-s)}{t}} (y-x_0 - \frac{s}{t} (x-x_0)) \frac{\Z(x_0;s,y)}{p_s(y-x_0)} W(ds,dy)\,,
		\end{align*}
		Then we obtain
		\begin{align*}
		&\left\|\frac{\Z(x_0;t,x)}{p_t(x-x_0)} - \frac{\Z_{\epsilon}(x_0;t,x)}{ p_t(x-x_0)} \right\|_{L^m(\Omega)} \\
		\leq& \left\| \int_0^t \int_{\RR^\ell} p_{\frac{s(t-s)}{t}} (y-x_0-\frac{s}{t} (x-x_0)) \frac{\Z(x_0;s,y) - \Z_{\epsilon}(x_0;s,y)}{p_s(y-x_0)} W(ds,dy) \right\|_{L^m(\Omega)}\\
		&+ \left\| \int_0^t \int_{\RR^\ell} p_{\frac{s(t-s)}{t}} (y-x_0-\frac{s}{t} (x-x_0)) \frac{\Z_{\epsilon}(x_0;s,y)}{p_s(y-x_0)} \left[W(ds,dy) - W_{\epsilon}(ds,dy)\right]  \right\|_{L^m(\Omega)}\\
		:=& I_1 + I_2\,.
		\end{align*}
		To estimate $I_1$, we use Lemma \ref{lem:Wm} to obtain
		\begin{align*}
		I_1 
		& \lesssim \sqrt{m} \left( \int_0^t \int_{\RR^\ell} e^{-\frac{2s(t-s)}{t} |\xi|^2} \mu(\xi)d\xi M^2_s ds \right)^{1/2}\\
		&\lesssim \sqrt{m} \left( \int_0^t \left(\frac{s(t-s)}{t}\right)^{-\frac{\alpha}{2}} M^2_s ds \right)^{1/2}\,.
		\end{align*}
		To estimate $I_2$, we first note that the noise $W-W_ \epsilon$ has spectral density $(1- e^{-\epsilon|\xi|^2})^2\mu(\xi)$. Applying Lemma \ref{lem:Wm}, we obtain
		\begin{align*}
			I_2\lesssim\sqrt{m} \sup_{s\le t,y\in\RR^\ell}\lt\|\frac{\Z_\epsilon(x_0;s,y)}{p_s(y-x_0)} \rt\|_{L^m(\Omega)} \left( \int_0^t \int_{\RR^\ell} e^{-\frac{2s(t-s)}{t} |\xi|^2}(1-e^{-\epsilon|\xi|^2})^2 \mu(\xi)d\xi  ds \right)^{1/2}\,.
		\end{align*}
		Let us fix $\beta\in(0,4\wedge(\ell- \alpha))$. Applying the elementary inequality $1-e^{-\epsilon|\xi|^2}\le \epsilon^{\beta/4}|\xi|^{\beta/2}$ together with the estimate
		\begin{equation*}
			\int_0^t\int_{\RR^\ell} e^{-\frac{2s(t-s)}{t} |\xi|^2}|\xi|^{\beta} \mu(\xi)d\xi ds\lesssim \int_0^t\lt(\frac{s(t-s)}t \rt)^{-\frac{\alpha+\beta}2}ds\lesssim t^{\frac{2- \alpha- \beta}2}\,,
		\end{equation*}
		we get
		\begin{align*}
			I_2\lesssim \epsilon^{\frac \beta 4} t^{\frac{2- \alpha-\beta}4}\sqrt m \sup_{s\le t,y\in\RR^\ell}\lt\|\frac{\Z_\epsilon(x_0;s,y)}{p_s(y-x_0)} \rt\|_{L^m(\Omega)}\,.
		\end{align*}
		Reasoning as in \cite[Lemma 4.1]{HLN15}, we see that
		\begin{multline*}
			\EE_B\exp\lt\{\sum_{1\le j<k\le m}\int_0^t \gamma_\epsilon(B^j_{0,t}(s)-B^k_{0,t}(s))ds \rt\}
			\\\le \EE_B\exp\lt\{\sum_{1\le j<k\le m}\int_0^t \gamma(B^j_{0,t}(s)-B^k_{0,t}(s))ds \rt\}\,.
		\end{multline*}
		Two key observations here are $\gamma_ \epsilon,\gamma$ have spectral measures $\mu(\xi),e^{-\epsilon|\xi|^2}\mu(\xi)$ respectively and $e^{-\epsilon|\xi|^2}\mu(\xi)\le \mu(\xi)$. Hence, it follows from \eqref{h1:supZm} and the previous estimate that
		\begin{equation*}
			\sup_{s\le t,y\in\RR^\ell}\lt\|\frac{\Z_\epsilon(x_0;s,y)}{p_s(y-x_0)} \rt\|_{L^m(\Omega)}\le 
			\Theta_t^{\frac1m}(m)\,.
		\end{equation*}
		In summary, we have shown
		\begin{equation*}
			M_t\lesssim \sqrt m\lt(\int_0^t\lt(\frac{s(t-s)}t \rt)^{-\frac \alpha2}M_s^2 ds \rt)^{\frac12}+\epsilon^{\frac \beta4}t^{\frac{2- \alpha -\beta}4}\sqrt m \Theta_t^{\frac1m}(m)\,.
		\end{equation*}
		Applying Lemma \ref{lem:gronwall}, this yields
		\begin{equation}
			M_t\lesssim \epsilon^{\frac \beta4}t^{\frac{2- \alpha-\beta}4}\sqrt m \Theta_t^{\frac1m}(m)e^{cm^{\frac2{2- \alpha}}t}\quad\mbox{for all}\quad t\ge0\,,
		\end{equation}
		for some constant $c$ depending only on $\alpha$.
	\end{proof}
	
	We are now ready to derive Feynman-Kac formulas for positive moments. 
	\begin{proposition}\label{prop:upB} Let $u_0$ be a measure satisfying \eqref{con:u0}. Under \ref{Cov:bounded} or \ref{Cov:Dalang}, for every $x_1,\dots,x_m\in\RR^\ell$, we have
		\begin{align}\label{eqn:FKbridge}
			\EE\left[\prod_{j=1}^m u(t,x_j)\right]
			&=\int_{(\RR^\ell)^m}   \notag 
			\EE\exp\left\{\int_0^t\sum_{1\le j<k\le m} \gamma\left(B_{0,t}^j(s)-B_{0,t}^k(s)+x_j-x_k+\frac st(y_j-y_k)\right)ds\right\}  \\
			&\quad  \times \prod_{j=1}^m [p_t(y_j)u_{0}(x_j+dy_j)]\,.
		\end{align}
	and
		\begin{equation}\label{ineq:upB}
		\EE \lt[\prod_{j=1}^m \frac{u(t,x_j)}{p_t*|u_0|(x_j)} \rt]\le \EE\exp\left\{\int_0^t\sum_{1\le j<k\le m} \gamma\left(B_{0,t}^j(s)-B_{0,t}^k(s)\right)ds\right\}\,.
	\end{equation}
	\end{proposition}
	\begin{proof} We prove the result under the hypothesis \ref{Cov:Dalang}. The proof under hypothesis \ref{Cov:bounded} is easier and omitted. 

		\textbf{Step 1:} we first consider \eqref{eqn:FKbridge} and \eqref{ineq:upB} when the initial data are Dirac masses. More precisely, we will show that
		\begin{multline}\label{tmp:ZmFK}
			\EE\left[\prod_{j=1}^m \frac{\Z(z_j;t,x_j)}{p_t(x_j-z_j)} \right]
			\\=  
			\EE\exp\left\{\int_0^t\sum_{1\le j<k\le m} \gamma\left(B_{0,t}^j(s)-B_{0,t}^k(s)+\frac st(z_j-z_k) +\frac{t-s}t(x_j-x_k)\right)ds\right\}\,.
		\end{multline}
		and
		\begin{equation}\label{tmp:supZm}
			\EE\left[\prod_{j=1}^m \frac{\Z(z_j;t,x_j)}{p_t(x_j-z_j)} \right]
			\le  
			\EE\exp\left\{\int_0^t\sum_{1\le j<k\le m} \gamma\left(B_{0,t}^j(s)-B_{0,t}^k(s)\right)ds\right\}\,.
		\end{equation}
		Fix $\epsilon>0$, \eqref{tmp:ZmFK} with $\Z,\gamma$ replaced by $\Z_\epsilon,\gamma_\epsilon$ has been obtained in \eqref{h1:ZmFK}. Namely, we have
		\begin{multline}\label{tmp.uep:mFK}
			\EE\left[\prod_{j=1}^m \frac{\Z_ \epsilon(z_j;t,x_j)}{p_t(x_j-z_j)} \right]
			\\=  
			\EE\exp\left\{\int_0^t\sum_{1\le j<k\le m} \gamma_ \epsilon\left(B_{0,t}^j(s)-B_{0,t}^k(s)+\frac st(z_j-z_k) +\frac{t-s}t(x_j-x_k)\right)ds\right\}
		\end{multline}
		Using analogous arguments with \cite[Proposition 4.2]{HLN15}, we can show that for every $\kappa\in\RR$, as $\epsilon\downarrow0$, the functions
		\begin{multline*}
			(x_1,z_1,\dots,x_m,z_m)\mapsto
			\\\EE\exp\lt\{\kappa\int_0^t \sum_{1\le j<k\le m}\gamma_ \epsilon\left(B_{0,t}^j(s)-B_{0,t}^k(s)+\frac st(z_j-z_k) +\frac{t-s}t(x_j-x_k)\right)ds \rt\}
		\end{multline*}
		converges uniformly on $\RR^{2m\ell}$ to the function
		\begin{multline*}
			(x_1,z_1,\dots,x_m,z_m)\mapsto
			\\\EE\exp\lt\{\kappa\int_0^t \sum_{1\le j<k\le m}\gamma\left(B_{0,t}^j(s)-B_{0,t}^k(s)+\frac st(z_j-z_k) +\frac{t-s}t(x_j-x_k)\right)ds \rt\}\,.
		\end{multline*}
		In addition, in view Proposition \ref{prop:ZZep}, 
		\begin{equation*}
			\lim_{\epsilon\downarrow0}\EE\left[\prod_{j=1}^m \frac{\Z_ \epsilon(z_j;t,x_j)}{p_t(x_j-z_j)} \right]=\EE\left[\prod_{j=1}^m \frac{\Z(z_j;t,x_j)}{p_t(x_j-z_j)} \right]\,.
		\end{equation*}
		Sending $\epsilon\downarrow0$ in \eqref{tmp.uep:mFK}, we obtain \eqref{tmp:ZmFK}. \eqref{tmp:supZm} is obtained analogously using \eqref{h1:supZm}. We omit the details.

		\textbf{Step 2:} For general initial data satisfying \eqref{con:u0}, we note that from \eqref{rep:u}, 
		\begin{equation*}
			\prod_{j=1}^m u(t,x_j)=\int_{(\RR^\ell)^m}  \prod_{j=1}^m [\Z(z_j;t,x_j) u_{0}(dz_j)]\,.
		\end{equation*}
		From here, it is evident that \eqref{eqn:FKbridge}, \eqref{ineq:upB} are consequences of \eqref{tmp:ZmFK}, \eqref{tmp:supZm} and Fubini's theorem. 
	\end{proof}
	We conclude this section with the following observation.
	\begin{remark}\label{rmk:nonnegative}
		Under \ref{Cov:bounded}, it is evident from \eqref{eqn:FKZ} that $\Z(z;t,x) $ is non-negative for every $z,t,x$. Under \ref{Cov:Dalang}, thanks to Proposition \ref{prop:ZZep}, $\Z(z;t,x)$ is the limit of non-negative random variables, hence $\Z(z;t,x) $ is also non-negative for every $z,t,x$. Furthermore, in view of \eqref{rep:u}, if $u_0$ is non-negative then $u(t,x)$ is non-negative for every $t,x$.
	\end{remark}

\section{Moment asymptotic and regularity} 
\label{sec:moment_asymptotic_and_regularity}
\subsection*{Moment asymptotic} 
	We begin with a study on high moments. Under hypothesis \ref{Cov:bounded}, the high moment asymptotic is governed by the value of $\gamma$ at the origin.
	\begin{proposition}\label{prop:H1m}
		Under \ref{Cov:bounded}, for every $T>0$, we have 
		\begin{equation}\label{lim:H1moment}
		 \limsup_{m \to \infty} m^{-2} \log\sup_{0<t\le T} \sup_{x\in \RR}  \EE \left(\frac{\Z(x_0;t,x)}{p_t(x-x_0)}\right)^m 
		 \le  \frac{T}{2}\gamma(0)\,.
		\end{equation}
	\end{proposition}
	\begin{proof}
	  Since $\gamma$ is positive definite, $\gamma(x)\leq \gamma(0)$ for all $x \in \RR^{\ell}$. It follows from \eqref{ineq:upB} that 
	  \begin{eqnarray*}
	  \EE \left( \frac{\Z(x_0;t,x)}{p_t(x-x_0)}\right)^m \leq \exp \left( \frac{m(m-1)}{2} t \gamma(0) \right) \,.
	  \end{eqnarray*}
	  This immediately yields \eqref{lim:H1moment}.
	\end{proof}
	The following result is needed to obtain moment asymptotic under \ref{Cov:Dalang}.
	\begin{lemma}\label{lem:limtaum} Suppose that $\mu(\RR^\ell)<\infty$. For each $t,T,m$ we put $t_m = m^{\frac{2}{2-\alpha}}t$ and $T_m = m^{\frac{2}{2-\alpha}}T$. Then
		\begin{equation}\label{eqn:mmtau}
	\limsup_{m \to \infty} \frac{1}{mT_m} \log \sup_{0 \leq t_m \leq T_m} \EE \exp \left( \frac{1}{m} \sum_{1 \leq j < k \leq m} \int_0^{t_m} \gamma\left( B_{0, t_m}^j(s)-B_{0, t_m}^k(s) \right) ds \right) \leq \frac{1}{2} \mathcal{E}_H(\gamma)\,.		
		\end{equation} 
	\end{lemma}
	\begin{proof}
		For each $\lambda\in(0,1)$, we note that
		\begin{multline*}
			\EE \exp \left\{ \frac{1}{m} \sum_{1 \leq j < k \leq m} \int_0^{t_m} \gamma\left( B_{0, t_m}^j(s)-B_{0, t_m}^k(s) \right) ds \right\}
			\\\le
			e^{\frac{(m-1)t_m} 2\gamma(0)(1- \lambda)} \EE \exp \left\{ \frac{1}{m} \sum_{1 \leq j < k \leq m} \int_0^{\lambda t_m} \gamma\left( B_{0, t_m}^j(s)-B_{0, t_m}^k(s) \right) ds \right\}\,.
		\end{multline*}
		Using \eqref{id.BBBM}, we see that the expectation above is at most
		\begin{align*}
			(1- \lambda)^{-\frac{m\ell}2}\EE\exp\lt\{\frac1m\sum_{1\le j<k\le m}\int_0^{\lambda t_m}\gamma(B^j(s)-B^k(s))ds\rt\}\,.
		\end{align*}
		In addition, reasoning as in \cite[Lemma 4.1]{HLN15}, we see that
		\begin{multline*}
			\sup_{0 \leq t_m \leq T_m}\EE \exp \left\{ \frac{1}{m} \sum_{1 \leq j < k \leq m} \int_0^{\lambda t_m} \gamma\left( B^j(s)-B^k(s) \right) ds \right\}
			\\=\EE \exp \left\{ \frac{1}{m} \sum_{1 \leq j < k \leq m} \int_0^{\lambda T_m} \gamma\left( B^j(s)-B^k(s) \right) ds \right\}\,.
		\end{multline*}
		It follows that
		\begin{multline*}
			\limsup_{m\to\infty}\frac1{mT_m} \log \sup_{0 \leq t_m \leq T_m} \EE \exp \left\{ \frac{1}{m} \sum_{1 \leq j < k \leq m} \int_0^{t_m} \gamma\left( B_{0, t_m}^j(s)-B_{0, t_m}^k(s) \right) ds \right\}
			\\\le \frac{1- \lambda}2\gamma(0)+  \limsup_{m\to\infty}\frac1{mT_m}\EE\exp\lt\{\frac{1}{m} \sum_{1 \leq j < k \leq m} \int_0^{\lambda T_m} \gamma\left( B^j(s)-B^k(s) \right) ds \rt\}\,.
		\end{multline*}
		Applying \cite{ChPh15}*{Theorem 1.1}, we get 
		\begin{align*}
		\limsup_{m\to\infty}\frac1{m \lambda T_m}\EE\exp\lt\{\frac{1}{m} \sum_{1 \leq j < k \leq m} \int_0^{\lambda T_m} \gamma\left( B^j(s)-B^k(s) \right) ds \rt\}\leq \frac{1}{2} \mathcal{E}_H(\gamma)\,.
		\end{align*}
		Thus we have shown
		\begin{multline*}
		\limsup_{m \to \infty} \frac{1}{mT_m} \log \sup_{0 \leq t_m \leq T_m} \EE \exp \left( \frac{1}{m} \sum_{1 \leq j < k \leq m} \int_0^{t_m} \gamma\left( B_{0, t_m}^j-B_{0, t_m}^k(s) \right) ds \right) 
		\\\leq \frac{\lambda}{2} \mathcal{E}_H(\gamma) + \frac{1- \lambda}{2}\gamma(0)\,.
		\end{multline*}
		Finally, we send $\lambda \to 1^-$ to finish the proof. 
	\end{proof}

	\begin{proposition}\label{prop:H2m}
		Assuming \ref{Cov:Dalang}, for every fixed $T>0$, 
		\begin{equation}\label{id:ma}
			\lim_{m\to\infty}m^{-\frac{4-\alpha}{2-\alpha}}\log\sup_{0<t\le T}\sup_{x\in\RR^\ell}\EE\lt(\frac{\Z(x_0;t,x)}{p_t(x-x_0)}\rt)^m\leq\frac T2\cee_H(\gamma)
		\end{equation}
		where $\cee_H(\gamma)$ is the Hartree energy defined in \eqref{eqn:ceeg}.
	\end{proposition}
	\begin{proof}
		 Applying inequality \eqref{ineq:upB}, we have
		\begin{equation*}
			\sup_{x\in\RR^\ell}\EE\lt(\frac{\Z(x_0;t,x)}{p_t(x-x_0)}\rt)^m\le \EE\exp\lt\{\int_0^t\sum_{1\le j<k\le m}\gamma(B_{0,t}^j(s)-B_{0,t}^k(s))ds \rt\}\,.
		\end{equation*}
		In addition, by the change of variable $s \rightarrow s m^{-\frac{2}{2- \alpha}}$ and the scaling property of Brownian bridge, $\{B_{0,\lambda t}(\lambda s),s\in[0, t]\} \stackrel{\mbox{law}}{=} \{\sqrt{\lambda} B_{0,t}(s),s\in[0,t]\}$, the right hand side in the above expression is the same as
		\begin{equation*}
		 	\EE\exp\lt\{\frac1m\int_0^{m^{\frac2{2-\alpha}}t}\sum_{1\le j<k\le m}\gamma\left(B_{0,m^{\frac2{2-\alpha}}t}^j(s)-B_{0,m^{\frac2{2-\alpha}}t}^k(s)\right)ds \rt\}\,.
		\end{equation*}
		Hence, denoting $t_m=m^{\frac2{2- \alpha}}t$ and $T_m=m^{\frac2{2- \alpha}}T$, we see that \eqref{id:ma} is equivalent to the statement
		\begin{multline}\label{tmp1}
			\limsup_{m\to\infty}\frac1{mT_m} \log\sup_{0<t_m\le T_m}\EE\exp\lt\{\frac1m\int_0^{t_m}\sum_{1\le j<k\le m}\gamma(B_{0,t_m}^j(s)-B_{0,t_m}^k(s))ds \rt\}
			\\\le \frac 12 \cee_H(\gamma)\,.
		\end{multline}
		Let $p,q>1$ such that $p^{-1}+q^{-1}=1$. By H\"older inequality
		\begin{align*}
		 	\EE\exp\lt\{\frac1m\int_0^{t_m}\sum_{1\le j<k\le m}\gamma(B_{0,t_m}^j(s)-B_{0,t_m}^k(s))ds \rt\}
		 	\le \A^{\frac1p}\B^{\frac1q}
		\end{align*} 
		where
		\begin{align*}
			\A&=\sup_{0<t_m\le T_m} \EE\exp\lt\{\frac pm\int_0^{t_m}\sum_{1\le j<k\le m}\gamma_ \epsilon(B_{0,t_m}^j(s)-B_{0,t_m}^k(s))ds \rt\}
			\\\B&=\sup_{0<t_m\le T_m}\EE\exp\lt\{\frac qm\int_0^{t_m}\sum_{1\le j<k\le m}(\gamma- \gamma_\epsilon)(B_{0,t_m}^j(s)-B_{0,t_m}^k(s))ds \rt\}\,.
		\end{align*}
		From Lemma \ref{lem:limtaum} and the fact that  $\cee_H(\gamma_\epsilon) \leq \cee_H(\gamma)$ (see \eqref{eqn:cee}), we have
		\begin{equation*}
			\lim_{p\to1^+}\limsup_{m\to\infty}\frac1{mT_m}\log \A\le \frac12\cee_H(\gamma)\,.
		\end{equation*}
		Hence, it suffices to show for every fixed $q>1$,
		\begin{equation}\label{eq: suppose we know}
			\lim_{\epsilon\downarrow0}\limsup_{m\to\infty}\frac1{mT_m}\log\B=0\,.
		\end{equation}
	By  Cauchy-Schwarz inequality and the fact that $B_{0,t} {\stackrel {\text{law}} {=}} B_{0,t}(t-\cdot)$, we have 
	\begin{align*}
	&\EE \exp \left\{ \frac{q}{m} \int_0^{t_m} \sum_{1 \leq j < k \leq m} \left( \gamma - \gamma_{\epsilon} \right) \left( B_{0, t_m}^j(s)- B_{0, t_m}^k(s) \right) ds \right\}\\
	&\quad\leq \EE \exp \left\{ \frac{2q}{m} \int_0^{\frac{t_m}{2}} \sum_{1 \leq j < k \leq m} \left( \gamma - \gamma_{\epsilon} \right) \left( B_{0, t_m}^j(s)- B_{0, t_m}^k(s) \right) ds \right\}\,,
	\end{align*}
	Together with \eqref{id.BBBM}, we arrive at 
	\begin{align*}
	&\EE \exp \left\{ \frac{q}{m} \int_0^{t_m} \sum_{1 \leq j < k \leq m} \left( \gamma - \gamma_{\epsilon} \right) \left( B_{0, t_m}^j(s)- B_{0, t_m}^k(s) \right) ds \right\}\\
	&\quad \leq 2^{m \ell} \EE \exp \left\{ \frac{2q}{m} \int_0^{\frac{t_m}{2}} \sum_{1 \leq j < k \leq m} \left( \gamma - \gamma_{\epsilon} \right) \left( B^j(s)- B^k(s) \right) ds \right\}\,,
		\end{align*}	
	 note that the right hand side of the above inequality is the $m$-th moment of the solution to the equation \eqref{eqn:SHE} driven by the noise with spatial covariance $\frac{2q}{m}\left(\gamma - \gamma_{\epsilon}\right)$, i.e., $\EE u(\frac{t_m}{2}, x)^m$,  the initial condition is $u_0(x) \equiv 2^{\ell}$. Using the hyper-contractivity as in \cites{HLN15,MR3531492}, we have 
	 \begin{align*}
	& \EE \exp \left\{ \frac{2q}{m} \int_0^{\frac{t_m}{2}} \sum_{1 \leq j < k \leq m} \left( \gamma - \gamma_{\epsilon} \right) \left( B^j(s)- B^k(s) \right) ds \right\}\\
	&\quad \leq \left[\EE \exp \left\{ \frac{2q(m-1)}{m} \int_0^{\frac{t_m}{2}}  \left( \gamma - \gamma_{\epsilon} \right) \left( B^1(s)- B^2(s) \right) ds \right\}\right]^{\frac{m}{2}} \\
	& \quad \leq \left[\sum_{k=0}^{\infty} (2q)^k \int_{[0, \frac{t_m}{2}]^k_{<}} \int_{\RR^{\ell k}} \prod_{j=1}^k \left(e^{-|\eta_j|^2 } (s_{j+1}-s_j)^{-\frac{\alpha}{2}}\right) \prod_{j=1}^k \left(1-e^{-\epsilon (s_{j+1} - s_j)^{-1} |\eta_j|^2} \right) \mu(\eta)d\eta ds \right]^{\frac{m}{2}}
	 \end{align*}
	where in the last line we have used the estimate (3.7) in \cite{MR3354615} and $\mu(\eta)d \eta$  is abbreviation for $\prod_{j=1}^k \mu(\eta_j)d\eta_j$. Since $\alpha < 2$, we can find a $\beta>0$ such that $\beta < 1-\frac{\alpha}{2}$. Then using the elementary inequality
	\begin{equation*}
	1-e^{-x} \leq  C_{\beta}x^{\beta}\quad\forall x>0\,,
	\end{equation*}
	we obtain 
	\begin{align*}
	&\sum_{k=0}^{\infty} (2q)^k \int_{[0, \frac{t_m}{2}]^k_{<}} \int_{\RR^{\ell k}} \prod_{j=1}^k \left(e^{-|\eta_j|^2 }  (s_{j+1}-s_j)^{-\frac{\alpha}{2}}\right) \prod_{j=1}^k \left(1-e^{-\epsilon (s_{j+1} - s_j)^{-1} |\eta_j|^2} \right) \mu(\eta)d\eta ds\\
	&\leq \sum_{k=0}^{\infty} (C_{\beta}2q\epsilon^{\beta})^k \int_{[0, \frac{t_m}{2}]^k_{<}} \int_{\RR^{\ell k}} \prod_{j=1}^k \left( e^{-|\eta_j|^2 } |\eta_j|^{2\beta}\right) (s_{j+1}-s_j)^{-\frac{\alpha}{2}-\beta}  \mu(\eta)d\eta ds\\
	& \leq \sum_{k=0}^{\infty} \frac{(Cq)^k t_m^{(-\frac{\alpha}{2} -\beta +1) k } \epsilon^{k\beta}}{\Gamma((-\frac{\alpha}{2} -\beta +1) k + 1)} \le C \exp \left(c t_m \epsilon^{\frac{\beta}{ -\frac{\alpha}{2} -\beta +1}}  \right)\,.
	\end{align*}
	Hence, we have shown
	\begin{align*}
	\B\leq C^m \exp \left( m T_m \epsilon^{\frac{\beta}{ -\frac{\alpha}{2} -\beta +1}}  \right)\,,
	\end{align*}
	from which \eqref{eq: suppose we know} follows. The proof for \eqref{id:ma} is complete. 
	\end{proof}

\subsection*{H\"older continuity} 
    We investigate the regularity of the process $\frac{\Z(x;t,y)}{p_t(y-x)}$ in the variables $x$ and $y$. These properties will be used in the proof of upper bound. For each integer $m\ge2$ and $t>0$, we recall that $\Theta_t(m)$ is defined in \eqref{def:Thetam}. 
	
	Note that from Proposition \ref{prop:upB}, we have
	\begin{equation}\label{ineq:T1}
		\sup_{s\in(0, t]}\sup_{x,y_1,\dots,y_{m}\in\RR^\ell}\EE\prod_{j=1}^{m}\frac{\Z(x;s,y_j)}{p_s(y_j-x)}=\Theta_t(m)\,.
	\end{equation}
	\begin{lemma}\label{lem:pa}
		For every $r>0$ and $y_1,y_2\in\RR^\ell$
		\begin{equation*}
			\||p_r(\cdot-y_1)-p_r(\cdot-y_2)|\|^2_{\HH_0}\le C r^{-\frac \alpha2}\lt(\frac{|y_2-y_1|}{r^{1/2}}\wedge 1 \rt)
		\end{equation*}
		under  \ref{Cov:Dalang}; and
		\begin{equation*}
			\||p_r(\cdot-y_1)-p_r(\cdot-y_2)|\|^2_{\HH_0}\le C \lt(\frac{|y_2-y_1|^2}{r}\wedge 1 \rt)
		\end{equation*}
		under \ref{Cov:bounded}.
		In the above, the constant $C$ does not depend on $y_1,y_2$ nor $r$.
	\end{lemma}
	\begin{proof}
		We denote $f(\cdot)=|p_r(\cdot-y_1)-p_r(\cdot-y_2)|$. Assuming first \ref{Cov:Dalang}, we observe the following simple estimate
		\begin{equation*}
			\iint_{\RR^\ell\times\RR^\ell}f(y)f(z)\gamma(y-z)dydz\le\sup_{z\in\RR^\ell}|f*\gamma(z)| \int_{\RR^\ell}f(y)dy\,.
		\end{equation*}
		Noting that
		\begin{equation*}
			\sup_{z\in\RR^\ell}|f*\gamma(z)|\le 2\sup_{z\in\RR^\ell}|p_r*\gamma(z)|=2p_r*\gamma(0)\lesssim r^{-\frac \alpha2}
		\end{equation*}
		and
		\begin{equation}\label{tmp:intf}
			\int_{\RR^\ell}f(y)dy\lesssim \lt(\frac{|y_2-y_1|}{r^{1/2}}\wedge 1 \rt)\,,
		\end{equation}
		the result easily follows. Under \ref{Cov:bounded}, we used the following inequality
		\begin{equation*}
			\iint_{\RR^\ell\times\RR^\ell}f(y)f(z)\gamma(y-z)dydz\le \gamma(0)\lt(\int_{\RR^\ell}f(y)dy \rt)^2
		\end{equation*}
		together with \eqref{tmp:intf} to obtain the result.
	\end{proof}
	
	\begin{proposition}\label{prop:hder} Assuming \ref{Cov:bounded} or \ref{Cov:Dalang}. There exists a constant $\eta\in(0,1)$ such that for every compact set $K$ and every integer $m\ge2$, 
	\begin{equation}\label{eqn:hder}
		\sup_{w\in\RR^\ell} \lt\|\sup_{\substack{x_1,x_2\in K,\\ y\in B(w,1)}}
		\frac{\lt|\frac{\Z(x_1;t,y)}{p_t(y-x_1)}-\frac{\Z(x_2;t,y)}{p_t(y-x_2)}\rt|}{|x_2-x_1|^{\eta}} \rt\|_{L^m(\Omega)}\le c_K(t) m^{\frac12}[\Theta_t(m)]^{\frac1m}e^{cm^{\frac{2}{2-\bar{\alpha}}}}\,,
	\end{equation} 
	and
	\begin{equation}\label{eqn:hder y}
		\sup_{w, x\in\RR^\ell} \lt\|\sup_{ y_1,y_2\in B(w,1)}
		\frac{\lt|\frac{\Z(x;t,y_1)}{p_t(y_1-x)}-\frac{\Z(x;t,y_2)}{p_t(y_2-x)}\rt|}{|y_2-y_1|^{\eta}} \rt\|_{L^m(\Omega)}\le c_K(t) m^{\frac12}[\Theta_t(m)]^{\frac1m}\,,
	\end{equation} 	
	where $B(w,1)$ is the closed unit ball in $\RR^\ell$ centered at $w$. 
	In the above, the constant $c$ depends only on $\bar{\alpha}$ and $\eta$ and $c_K(t)$ depends only on $K,t, \eta$.
	\end{proposition}
	\begin{proof}
		We present the proof under hypothesis \ref{Cov:Dalang} in detail. The proof for the other case is similar and is omitted.	We first show that for every $\eta\in(0,2- \alpha)$,
		\begin{equation}\label{tmp:estZ1}
			\sup_{x\in\RR^\ell}\lt\|\frac{\Z(x_1;t,x)}{p_t(x-x_1)}-\frac{\Z(x_2;t,x)}{p_t(x-x_2)}\rt\|_{L^m(\Omega)}
			\lesssim_t \sqrt m[\Theta_t(m)]^{\frac1m}|x_2-x_1|^{\frac\eta2} e^{cm^{\frac{2}{2-\alpha}}}
		\end{equation}
		Fix $t>0$ and $x_1,x_2,x\in\RR^\ell$. From \eqref{eqn:Z}, we have
		\begin{multline}\label{tmp:Z1}
			\frac{\Z(x_1;t,x)}{p_t(x-x_1)}-\frac{\Z(x_2;t,x)}{p_t(x-x_2)}
			=\int_0^t\int_{\RR^\ell} f(s,y) \frac{\Z(x_1;s,y)}{p_s(y-x_1)} W(ds,dy)
			\\+\int_0^t\int_{\RR^\ell}p_{\frac{s(t-s)}t}\lt(y-x_2-\frac st(x-x_2) \rt)
			\lt[\frac{\Z(x_2;s,y)}{p_t(y-x_2)}-\frac{\Z(x_1;s,y)}{p_t(y-x_1)}\rt]W(ds,dy)
		\end{multline}
		where
		\begin{equation*}
			f(s,y)=p_{\frac{s(t-s)}{t}} \left(y-x_1-\frac{s}{t}(x-x_1)\right) - p_{\frac{s(t-s)}{t}}\left(y-x_2-\frac{s}{t}(x-x_2)\right)\,.
		\end{equation*}
		Obviously $f$ also depends on $t,x_1,x_2$ and $x$, however these parameters will be omitted.
		For each integer $m\ge2$, applying Lemma \ref{lem:Wm} we see that
		\begin{equation*}
			\lt\|\int_0^t\int_{\RR^\ell} f(s,y) \frac{\Z(x_1;s,y)}{p_s(y-x_1)} W(ds,dy)\rt\|_{L^m(\Omega)}
			\leq\sqrt{4m} [\Theta_t(m)]^{\frac1m}\||f(s,y)| \1_{[0,t]}(s) \|_{\HH_{s,y}} \,. 
		\end{equation*}
		Applying Lemma \ref{lem:pa}, for every $\eta\in(0,2- \alpha)$, there exists $c_{\eta}>0$ such that
		\begin{equation*}
			\||f(s,y)| \1_{[0,t]}(s) \|_{\HH_{s,y}}\le c_{\eta}t^{\frac12-\frac{\alpha+\eta}4}  |x_2-x_1|^{\frac\eta2}\,.
		\end{equation*}
		Hence,
		\begin{align}\label{tmp:extZ2}
			\lt\|\int_0^t\int_{\RR^\ell} f(s,y) \frac{\Z(x_1;s,y)}{p_s(y-x_1)} W(ds,dy)\rt\|_{L^m(\Omega)}
			\le c_{\eta}t^{\frac12-\frac{\alpha+\eta}4} \sqrt m[\Theta_t(m)]^{\frac1m}|x_2-x_1|^{\frac\eta2}\,.
		\end{align}
		For each $s>0$, we set
		\begin{equation*}
			M_s=\sup_{x\in\RR^\ell}\lt\|\frac{\Z(x_1;s,x)}{p_s(x-x_1)}-\frac{\Z(x_2;s,x)}{p_s(x-x_2)}\rt\|_{L^m(\Omega)}\,.
		\end{equation*}
		It follows from Lemma \ref{lem:Wm} that
		\begin{align*}
			&\lt\|\int_0^t\int_{\RR^\ell}p_{\frac{s(t-s)}t}\lt(y-x_2-\frac st(x-x_2) \rt)
			\lt[\frac{\Z(x_2;s,y)}{p_t(y-x_2)}-\frac{\Z(x_1;s,y)}{p_t(y-x_1)}\rt]W(ds,dy)\rt\|_{L^m(\Omega)}
			\\&\le c \sqrt{m}\lt( \int_0^t \lt\|p_{\frac{s(t-s)}t}\lt(\cdot-x_2-\frac st(x-x_2) \rt)\rt\|^2_{\HH_{0}}M_s^2ds\rt)^{\frac12}
			\\&=c\sqrt{m}\lt(\int_0^t \lt(\frac{s(t-s)}t \rt)^{-\frac \alpha2}M_s^2 ds \rt)^{\frac12}\,,
		\end{align*}
		where $c$ is some constant. Applying these estimates in \eqref{tmp:Z1} yields
		\begin{align*}
			M_t\le c_ \eta t^{\frac12-\frac{\alpha+\eta}2} \sqrt m[\Theta_t(m)]^{\frac1m}|x_2-x_1|^{\frac\eta2}+c\sqrt{m}\lt(\int_0^t \lt(\frac{s(t-s)}t \rt)^{-\frac \alpha2}M_s^2 ds \rt)^{\frac12}\,.
		\end{align*}
		We now apply Lemma \ref{lem:gronwall} to get
		\begin{equation*}
		 	M_t\lesssim_t \sqrt m[\Theta_t(m)]^{\frac1m}|x_2-x_1|^{\frac\eta2}e^{c m^{\frac{2}{2-\alpha}}}\,,
		\end{equation*} 
		which is exactly \eqref{tmp:estZ1}.
		
		To complete the proof of the estimate \eqref{eqn:hder}. Fix $t>0$ and $x_1,x_2,y_1,y_2\in\RR^\ell$. Observe that
		\begin{align*}
			&\frac{\Z(x_1;t,y_1)}{p_t(y_1-x_1)}-\frac{\Z(x_2;t,y_2)}{p_t(y_2-x_2)}
			\\&=\int_0^t\int_{\RR^\ell} g(s,y) \frac{\Z(x_1;s,y)}{p_s(y-x_1)} W(ds,dy)
			\\&\quad+\int_0^t\int_{\RR^\ell}p_{\frac{s(t-s)}t}\lt(y-x_2-\frac st(y_2-x_2) \rt)
			\lt[\frac{\Z(x_2;s,y)}{p_t(y-x_2)}-\frac{\Z(x_1;s,y)}{p_t(y-x_1)}\rt]W(ds,dy)
			\\&=I_1+I_2\,,
		\end{align*}
		where
		\begin{equation*}
			g(s,y)=p_{\frac{s(t-s)}{t}} \left(y-x_1-\frac{s}{t}(y_1-x_1)\right) - p_{\frac{s(t-s)}{t}}\left(y-x_2-\frac{s}{t}(y_2-x_2)\right)\,.
		\end{equation*}
		Similar to \eqref{tmp:extZ2}, we have
		\begin{equation}\label{eq:I1}
			\lt\|I_1\rt\|_{L^m(\Omega)}
			\lesssim_t\sqrt{m} [\Theta_t(m)]^{\frac1m}(|x_2-x_1|+|y_2-y_1|)^{\frac \eta2} \,. 
		\end{equation}
		$I_2$ can be estimated using Lemma \ref{lem:Wm} and \eqref{tmp:estZ1}
		\begin{align*}
			\|I_2\|_{L^m(\Omega)}\lesssim_t \sqrt{m} [\Theta_t(m)]^{\frac1m}|x_2-x_1|^{\frac \eta2} e^{c m ^{\frac{2}{2-\alpha}}}\,. 
		\end{align*}
		Hence, we have shown
		\begin{equation*}
			\lt\|\frac{\Z(x_1;t,y_1)}{p_t(y_1-x_1)}-\frac{\Z(x_2;t,y_2)}{p_t(y_2-x_2)}\rt\|_{L^m(\Omega)}\lesssim_t\sqrt{m} [\Theta_t(m)]^{\frac1m}(|x_2-x_1|+|y_2-y_1|)^{\frac \eta2} e^{cm^{\frac{2}{2-\alpha}}} \,.
		\end{equation*}
		At this point, the estimate \eqref{eqn:hder} follows from the Garsia-Rodemich-Rumsey inequality (cf. \cite{garsiarodemich}).
		
		The proof of \eqref{eqn:hder y} is simpler. Actually, by writing 
		\begin{align*}
		&\frac{\Z(x;t,y_1)}{p_t(y_1-x)}-\frac{\Z(x;t,y_2)}{p_t(y_2-x)}\\
				=&\int_0^t\int_{\RR^\ell} \left(p_{\frac{s(t-s)}{t}} \left(y-x-\frac{s}{t}(y_1-x)\right) - p_{\frac{s(t-s)}{t}}\left(y-x-\frac{s}{t}(y_2-x)\right) \right) \frac{\Z(x;s,y)}{p_s(y-x)} W(ds,dy)\,,
		\end{align*}	
		we get an estimate for $\|\frac{\Z(x;t,y_1)}{p_t(y_1-x)}-\frac{\Z(x;t,y_2)}{p_t(y_2-x)}\|_{L^m(\Omega)}$  as in \eqref{eq:I1}. The estimate 	\eqref{eqn:hder y} again follows from the Garsia-Rodemich-Rumsey inequality (cf. \cite{garsiarodemich}). We omit the details. 	
	\end{proof}

In proving \eqref{lim:H2ep}, we need to handle the asymptotic of $\sup_{\epsilon<1}\sup_ {x\in K,|y|\le R} \frac{\Z_{\epsilon}(x;t,y)}{p_t(y-x)}$, thus we write down the H\"older continuity result for $\frac{\Z_{\epsilon}(x;t,y)}{p_t(y-x)}$ with respect to $\epsilon,x,y$. The proof is similar with Proposition \ref{prop:hder} and is left to the reader. 
\begin{proposition}	
\label{prop:hlder epsilon} 
	Assuming \ref{Cov:bounded} or \ref{Cov:Dalang}. There exists a constant $\eta\in(0,1)$ such that for every compact set $K$ and every integer $m\ge2$, 
	\begin{equation}\label{eqn:hder epsilon}
		\sup_{w\in\RR^\ell} \lt\|\sup_{\substack{x_1,x_2\in K, y\in B(w,1)\\ \epsilon, \epsilon' \in (0,1]}}
		\frac{\lt|\frac{\Z_{\epsilon}(x_1;t,y)}{p_t(y-x_1)}-\frac{\Z_{\epsilon'}(x_2;t,y)}{p_t(y-x_2)}\rt|}{(|x_2-x_1|+ |\epsilon- \epsilon'|)^{\eta}} \rt\|_{L^m(\Omega)}\le c_K(t) m^{\frac12}[\Theta_t(m)]^{\frac1m}e^{cm^{\frac{2}{2-\bar{\alpha}}}}\,,
	\end{equation} 
	and
	\begin{equation}\label{eqn:hder y epsilon}
		\sup_{w, x\in\RR^\ell;\epsilon\le1} \lt\|\sup_{ y_1,y_2\in B(w,1)}
		\frac{\lt|\frac{\Z_{\epsilon}(x;t,y_1)}{p_t(y_1-x)}-\frac{\Z_{\epsilon}(x;t,y_2)}{p_t(y_2-x)}\rt|}{|y_2-y_1|^{\eta}} \rt\|_{L^m(\Omega)}\le c_K(t) m^{\frac12}[\Theta_t(m)]^{\frac1m}\,.
	\end{equation} 	
	In the above, the constant $c$ depends only on $\bar{\alpha}$ and $\eta$ and $c_K(t)$ depends only on $K,t, \eta$.
\end{proposition}

\section{Spatial asymptotic}\label{sec:upper}
In this section we study the asymptotic 
of
\begin{equation*}
 	\sup_{|y|\le R}\frac{u(t,y)}{p_t*u_0(y)}
\end{equation*}
as described in Theorems \ref{thm:H1Z}, \ref{thm:H23Z} and \ref{thm:uep}. In what follows, we denote
\begin{equation}
 	a=\frac{2}{4-\bar \alpha}
 \end{equation} 
where we recall that $\bar{\alpha}$ is defined in \eqref{eq:unify notation}. Since $0\le\bar\alpha<2$, $a$ ranges inside the interval $[1/2,1)$. Because $R\mapsto \sup_{|y|\le R}\frac{u(t,y)}{p_t*u_0(y)}$ is monotone, it suffices to show these results along lattice sequence $R\in\{e^n\}_{n\ge1}$.

\subsection{The upper bound}
This subsection is devoted to the proof of upper bounds in Theorems \ref{thm:H1Z} and \ref{thm:H23Z} by combining the moment asymptotic bounds and the regularity estimates obtained in Section \ref{sec:moment_asymptotic_and_regularity}. 
We also recall that $\Theta_t(m)$ is defined in \eqref{def:Thetam}. {Propositions \ref{prop:H1m}, \ref{prop:H2m} together with \eqref{ineq:T1} imply}
\begin{equation}\label{m.theta}
	\limsup_{m\to\infty}m^{-\frac{4- \bar\alpha}{2- \bar\alpha}}\log \Theta_t(m)\le \frac t2\cee\,,
\end{equation}
where $\cee$ is defined in \eqref{eq:unify notation}.
The following result gives an upper bound for spatial asymptotic of $\Z(x;\cdot,\cdot)$. 
	\begin{theorem}\label{thm:Z up}
		For every compact set $K$, we have
		\begin{equation}\label{Zup}
			\limsup_{n\to\infty}n^{-a}\sup_{x\in K,|y|\le e^n}\lt(\log \Z(x;t,y)+\frac{|y-x|^2}{2t} \rt)\le \frac{4- \bar\alpha}2\ell^{\frac{2}{4- \bar\alpha}}\lt(\frac{\cee}{2- \bar\alpha}t\rt)^{1-a}
		\end{equation}
	\end{theorem}
	\begin{proof}
	We begin by noting that according Remark \ref{rmk:nonnegative}, $\Z(x;t,y)$ is non negative a.s. for each $x,y,t$. Let $t$ be fixed and put
	\begin{equation*}
		\K(x,y)=\frac{\Z(x;t,y)}{p_t(y-x)}\,,
	\end{equation*}
	where we have omitted the dependence on $t$.
	For every $n>1$ and every $\lambda>0$, we consider the probability
	\begin{equation*}
		P_n:=P\left(\sup_{x\in K,|y|\le e^n}\log\K(x,y)> \lambda n^a\right)\,.
	\end{equation*}
	Let $b$ be a fixed number such that $a < b <1$. We can find the points $x_i$, $i = 1, \dots, M_{n}$, such that $K \subset \cup_{i=1}^{M_n}B(x_i, e^{-n^b})$ and $M_n\lesssim e^{\ell n^b}$. In addition, by partitioning the ball $B(0,e^n)$ into unit balls, we see that $P_n$ is at most
	\begin{equation*}
		c(\ell) e^{\ell n+\ell n^b}\sup_{w\in\RR^\ell, x_i}P\left(\sup_{x\in B(x_i, e^{-n^b}), y\in B(w,1)}\K(x,y)> e^{\lambda n^a}\right)\,. 
	\end{equation*}
	Applying Chebychev inequality, we see that
	\begin{align*}
		P\left(\sup_{x\in B(x_i, e^{-n^b}),y\in B(w,1)}\K(x,y)> e^{\lambda n^a}\right)
		\le e^{-\lambda m n^a}\lt\|\sup_{x\in B(x_i, e^{-n^b}),y\in B(w,1)}\K(x,y) \rt\|_{L^m(\Omega)}^m\,.
	\end{align*}
	The above $m$-th moment is estimated by triangle inequality
	\begin{align*}
	\lt\|\sup_{{x\in B(x_i, e^{-n^b})},y\in B(w,1)}\K(x,y) \rt\|_{L^m(\Omega)}^m
	&\leq 3^m \left\|\sup_{{x\in B(x_i, e^{-n^b})},y\in B(w,1)} \left|\K(x,y)-\K(x_i,y)\right|\right\|_{L^m(\Omega)}^m\\
	&\quad+ 3^m \left\|\sup_{y\in B(w,1)} \left|\K(x_i,y)-\K(x_i,w)\right|\right\|_{L^m(\Omega)}^m\\
	& \quad + 3^m \left\|\K(x_i,w)\right\|_{L^m(\Omega)}^m\\
	&:=3^m( I_1 + I_2 + I_3)\,.
	\end{align*}
Using Proposition \ref{prop:hder} and \eqref{ineq:T1}, we see that 
\begin{equation}
I_1\lesssim  e^{-\eta m n^b+ c m ^{\frac1{1-a}}} \Theta_t(m)\,, \quad I_2 \lesssim   m^{\frac m2} \Theta_t(m)\,, \quad I_3\le  \Theta_t(m)\,.
\end{equation}
Altogether, we have
\begin{equation*}
P_n \lesssim 3^me^{\ell n^b+\ell n-\lambda m n^a}  \left(e^{-\eta m n^b+c m ^{\frac1{1-a} }} \Theta_t(m)+ m^{\frac m2} \Theta_t(m) \right)\,.
\end{equation*}
	For each $\beta>0$, we choose $m= \lfloor\beta n^{1-a} \rfloor$. In addition, for every fixed $\epsilon>0$,  \eqref{m.theta} yields
	\begin{equation*}
		\log\Theta_t(\lfloor\beta n^{1-a}\rfloor)\le \lt(\frac t2\cee+\epsilon\rt) \beta^{\frac{1}{1-a}}n
	\end{equation*}
	for all $n$ sufficiently large. It follows that
	\begin{align}\label{tmp.BCup}
		\sum_{n=1}^\infty P\left(\sup_{x\in K,|y|\le e^n}\log\K(x,y)> \lambda n^{a}\right) \lesssim S_1+S_2\,,
	\end{align} 
	where
	\begin{align*}
		S_1&=\sum_{n=1}^\infty\exp \left\{ \ell n^b+\beta(\log 3) n^{1-a} + (\ell- \lambda \beta+c \beta^{\frac1{1-a}})n   - \eta \beta n^{1-a+b}\right\}\,,
		\\S_2&=\sum_{n=1}^\infty\exp \left\{ -\ell n^b + n \ell - \lambda \beta n + \lt(\frac t2\cee+\epsilon\rt) \beta^{\frac1{1-a}}n \right\}\,.
	\end{align*}
	Since $1-a+b>1$, the term $e^{-\eta\beta n^{1-a+b}}$ is dominant, and hence,  $S_1$ is finite for every $\lambda,\beta>0$. To ensure the convergence of $S_2$, we choose $\lambda$ such that
	\begin{equation}\label{eqn:lambda}
		\lambda> \ell \beta^{-1}+(\frac t2\cee+\epsilon) \beta^{\frac a{1-a}}\,.
	\end{equation}
	It follows that the series on the right hand side of \eqref{tmp.BCup} is finite. By Borel-Cantelli lemma, we have almost surely
	\begin{equation*}
		\limsup_{n\to\infty} n^{-a}\sup_{x\in K,|y|\le e^n}\log\K(x,y)\le \lambda\,.
	\end{equation*}
	Evidently, the best choice for $\lambda$ is 
	\begin{align}
		\lambda_0:&=\inf_{\epsilon>0,\beta>0}\lt\{\ell \beta^{-1}+(\frac t2\cee+\epsilon) \beta^{\frac a{1-a}}\rt\}
		\notag\\&=\frac{4- \bar\alpha}2\ell^{\frac2{4- \bar\alpha}} \lt(\frac{t\cee}{2- \bar\alpha} \rt)^{\frac{2- \bar\alpha}{4- \bar\alpha}}\,,
	\end{align}
	which yields \eqref{Zup}.
	\end{proof}
	\begin{remark}\label{rmk: Z epsilon up}
	Using Proposition \ref{prop:hlder epsilon} and analogous arguments in Theorem \ref{thm:Z up}, we can show that 
	\begin{equation}
			\limsup_{R\to\infty}(\log R)^{-\frac2{4- \bar\alpha}}\sup_{x\in K, \epsilon \in (0,1], |y|\le R}\lt(\log \Z_{\epsilon}(x;t,y)+\frac{|y-x|^2}{2t} \rt)\le \frac{4- \bar\alpha}2\ell^{\frac2{4-\bar \alpha}}\lt(\frac{\cee}{2- \bar\alpha}t\rt)^{\frac{2- \bar\alpha}{4- \bar\alpha}}\,.
		\end{equation}
		We omit the details. 	
	\end{remark}

\subsection{The lower bound} 
\label{sub:low}
	We now focus on the lower bound of \eqref{coj:h1} and \eqref{lim:H2ep}. To start with, we explain an issue of using the localization procedure as in \cites{MR3474477,MR3098071}. In these papers, a localized version of the equation \eqref{eqn:SHE} is introduced, i.e.
	\begin{equation}\label{eq:U local}
	U^{\beta}(t,x) = 1 + \int_0^t \int_{|y-x|\leq \beta \sqrt{t}} p_{t-s}(x-y) U^{\beta}(s,y)W(ds,dy)\,,
	\end{equation}
for some $\beta>0$. For fixed $t$ and $\beta$ sufficiently large, $\sup_{|x|\le R} U^\beta(t,x)$ gives a good approximation for $\sup_{|x|\le R}u(t,x)$ as $R\to\infty$. In our situation, suppose for instance that $u_0=\delta(\cdot-x_0)$, the random field $\frac{\Z(x_0; t,x)}{p_t(x-x_0)}$ satisfies the equation 
\begin{equation}\label{tmp610}
\frac{\Z(x_0; t,x)}{p_{t}(x-x_0)} = 1 + \int_0^t \int_{\RR^{\ell}} p_{\frac{s(t-s)}{t}}\left( y-x_0-\frac{s}{t}(x-x_0) \right) \frac{\Z(x_0; s,y)}{p_s(y-x_0)}W(ds,dy)\,.
\end{equation}	
Since the  kernel $p_{\frac{s(t-s)}{t}}\left( y-x_0-\frac{s}{t}(x-x_0)\right)$ now involves $s$ and $t$ with $s$ moving from $0$ to $t$, the mass concentration of the stochastic integration on the right-hand side of \eqref{tmp610} varies and depends on $s$. We are not able to find a fixed localized integration domain similar as $\{y:|y-x|\leq \beta\sqrt{t}\}$. To get around this difficulty, we propose an alternative result (Theorem \ref{thm:uep}) which is about the regularized version of $\Z$, i.e., $\Z_{\epsilon}$. 
	To handle the spatial asymptotic of $\Z_ {\epsilon}$, we rely on the Feynman-Kac representation \eqref{eqn:FKZ} and adopt an argument developed by Xia Chen in \cite{MR3474477} with an additional scaling procedure. 

	Hereafter, $t$ and $\epsilon$ are  fixed positive constants, $n$ is the driving parameter which tends to infinity, 
		\begin{equation}\label{eq: ep R}
			\epsilon_n=
			\lt\{
			\begin{array}{ll}
				0& \text{if \ref{Cov:bounded} holds}
			 	\\\epsilon\lt(\frac t{n} \rt)^a &\text{if \ref{Cov:Dalang} holds}\,.
			\end{array} 
			\rt.
		\end{equation} 
		Let $y_1,\dots,y_N$ be $N$ points in $B(0,e^n)$ and $d$ be a positive number such that
		\begin{equation}\label{y.pts}
			N\lesssim e^{\ell n}\quad \mbox{and}\quad |y_j-y_k|\ge d\quad\forall j\neq k\,.
		\end{equation}
		Under \ref{Cov:bounded}, $d$ is chosen to be sufficiently large, depending on the shape of $\gamma$, while under \ref{Cov:Dalang}, we can simply choose $d=1$. See Lemma \ref{lem:H1} below for more details.
	\begin{theorem}\label{thm:low} For every $x_0\in\RR^\ell$
		\begin{equation}
			\liminf_{n\to\infty}n^{-a}\sup_{|y|\le e^n}\sup_{\epsilon\in(0,1)} \log \frac{\Z_\epsilon(x_0;t,y)}{p_t(y-x_0)}\ge \frac{4-\bar \alpha}2\ell^{\frac{2}{4-\bar \alpha}}\lt(\frac{\cee}{2-\bar \alpha}t \rt)^{\frac{2-\bar \alpha}{4-\bar \alpha}}
		\end{equation}
	\end{theorem}
\begin{proof} 
	\textbf{Step 1:} 
		Let $m=m_n$ be a natural number such that 
		\begin{equation}\label{eq: m R}
			{\lim_{n\to \infty} \frac{n^{1-a}}{m_n} \to 0}\,. 
		\end{equation} 
		Under hypothesis \ref{Cov:bounded}, for each $j$, we define the stopping time 
		\begin{equation}\label{def.tau1}
			\tau^j=\inf\lt\{s\ge 0:|B^j(s)|\ge r_0  \rt\}
		\end{equation}
		where $r_0>0$ is chosen so that 
		\begin{equation}\label{def.r0}
			\inf_{|x|< 2 r_0}\gamma(x)>0\,.
		\end{equation}
		Such a constant always exists since $\gamma$ is continuous and $\gamma(0)>0$. Under hypothesis \ref{Cov:Dalang}, the stopping times depends on $n$ and an arbitrary domain. More precisely, let  $D$ be an open bounded ball in $\RR^\ell$ which contains 0. For each $j$, $\tau^j=\tau^j_n(D)$ denotes the stopping time
		\begin{equation}\label{def.tau2}
		  	\tau_n^j(D)=\inf\lt\{s\ge 0:B^j(s)\not\in \lt(\frac{t}{n}\rt)^{\frac a2}  D \rt\}\,.
		\end{equation}
		As previously, we denote
		\begin{equation*}
			\K_ {\epsilon_n}(x,y)=\frac{\Z_ {\epsilon_n}(x;t,y)}{p_t(y-x)}\,, 
		\end{equation*}
		omitting the dependence on $t$.
		We note that from \eqref{eqn:FKZ} 
		\begin{align*}
			&\left(\K_{\epsilon_n}(x_0,y) \right)^m 
			\\&= \EE_B \exp \left( \sum_{j=1}^m \int_0^t \int_{\RR^{\ell}} \delta\left(B_{0,t}^j(t-s)+\frac{t-s}{t}x_0 + \frac{s}{t}y-z\right)W_ {\epsilon_n}(ds,dz) - \frac{tm}{2}\gamma_ {\epsilon_n}(0) \right)\\
			&= e^{-\frac{tm}2 \gamma_ {\epsilon_n}(0)} \EE_B e^{\xi_m(x_0,y)}, 
		\end{align*}
		where 
		\begin{equation}
			\xi_m(x_0,y)=\sum_{j=1}^m \int_{0}^t \int_{\RR^{\ell}} \delta\left(B_{0,t}^j(t-s)+\frac{t-s}tx_0 + \frac{s}{t}y-z\right)W_ {\epsilon_n}(ds,dz)\,.
		\end{equation}
		Conditioning on $B$, the variance of $\xi_m(x_0,y)$ is given by 
		\begin{equation*}
			S_m^2=\EE_B (\xi_m(x_0,y)^2)= \sum_{j,k=1}^m\int_{0}^t \gamma_ {\epsilon_n}(B_{0,t}^j(s)-B_{0,t}^k(s))ds\,.
		\end{equation*} 
				For every $\lambda\in(0,\sqrt{2\ell})$, it is evident that
		\begin{align*}
			\EE_B e^{\xi_m(x_0,y) }
			&\geq \EE_B \left \{ e^{\lambda \sqrt{n} S_m(t)}; \xi_m(x_0,y)\geq \lambda \sqrt{n} S_m(t), \min_{1\le k\le m} \tau^k\ge t  \right\}
			\\&=[\EE_B Z_m(n)]\eta_n(x_0,y)  \,,
		\end{align*}
		where we have put
		\begin{equation}
			Z_m(n)= e^{\lambda \sqrt{n} S_m(t) } \1 _{\{\min_{1\le j\le m} \tau^j_n(D)\ge t \}}\,,
		\end{equation}
		and
		\begin{equation}
			\eta_n(x_0,y) := \left[ \EE_B Z_m(n) \right]^{-1} \EE_B \left( Z_m(n) \1 _{\{ \xi_m(x_0,y) \ge \lambda \sqrt{n} S_m(t) \}} \right)\,.
		\end{equation}
		Combining all previous estimates, we arrive at an important inequality
		\begin{equation}\label{tmp:KZ}
			\K_{\epsilon_n} (x_0,y)\ge e^{-\frac t2 \gamma_{\epsilon_n}(0)}[\EE_B Z_m(n)]^{\frac1m}[\eta_n(x_0,y)]^{\frac1m}\,.
		\end{equation}
		It follows that
		\begin{align*}
			\sup_{j=1,\dots,N}\K_{\epsilon_n}(x_0,y_j)
			&\ge N^{-\frac1m}\lt(\sum_{j=1}^N [\K_{\epsilon_n}(x_0,y_j)]^m\rt)^{\frac1m}
			\\&\ge N^{-\frac1m}e^{-\frac t2 \gamma_{\epsilon_n}(0)}[\EE_B Z_m(n)]^{\frac1m} \lt(\sum_{j=1}^N\eta_n(x_0,y_j)  \rt)^{\frac1m}\,.
		\end{align*}
		We put
		\begin{equation}\label{eq:eta R}
			\eta^c_n(x_0)=\left[ \EE_B Z_m(n) \right]^{-1} \EE_B \left( Z_m(n) \1 _{\{ \max_{j=1,\dots,N}\xi_m(x_0,y_j) < \lambda \sqrt{n} S_m(t) \}} \right)\,.
		\end{equation}
		Applying the estimate
		\begin{align*}
			\sum_{j=1}^N\eta_n(x_0,y_j) 
			\ge1- \eta^c_n(x_0)\,,
		\end{align*}
		we obtain
		\begin{align}\label{tmp:Keta}
			\sup_{j=1,\dots,N}\K_{\epsilon_n}(x_0,y_j)
			\ge N^{-\frac1m}e^{-\frac t2 \gamma_{\epsilon_n}(0)}[\EE_B Z_m(n)]^{\frac1m} [1- \eta_n^c(x_0)]^{\frac1m}
		\end{align}
		Noting that $N^{-\frac1m} \lesssim e^{\ell \frac nm} $ and by \eqref{eqn:scalgamma}, $\gamma_{\epsilon_n}(0)=\epsilon_n^{-\frac \alpha2}\gamma_1(0)\lesssim n^{\frac \alpha2a}$, we see that
		\begin{equation}\label{eq: diag goto 0}
			\lim_{n\to\infty}n^{-a}\log\lt(N^{-\frac1m}e^{-\frac t2 \gamma_{\epsilon_n}(0)}\rt)=0\,.
		\end{equation}
		In other words, the factor $N^{-\frac1m}e^{-\frac t2 \gamma_{\epsilon_n(0)}}$ in \eqref{tmp:Keta} is negligible.
		In addition, we claim that for every $\lambda\in(0,\sqrt{2\ell})$ and every $x\in\RR^\ell$
		\begin{equation}\label{claim.eta}
		 	\lim_{n\to\infty}\eta_n^c(x_0)=0\quad \mathrm{a.s.}
		 \end{equation}
		We postpone the proof of this claim till Lemmas \ref{lem:H1} below. It follows that
		\begin{equation}\label{tmp:step2}
			\liminf_{n\to\infty}n^{-a}\log \max_{j=1, \dots,  N}\K_{\epsilon_n}(x_0,y_j)
			\ge \liminf_{n\to\infty}n^{-a}m^{-1}\log\EE_B Z_m(n)\,.
		\end{equation}
		
		\textbf{Step 2:} We will show that
		\begin{equation}\label{lim:ZmR}
				\liminf_{\epsilon\downarrow0,D\uparrow\RR^\ell}\liminf_{n\to\infty} n^{-a}m^{-1} \log\EE_BZ_m(n)\ge   \frac{4- \bar\alpha}4 \lambda^{\frac4{4- \bar\alpha}} \lt(\frac{2t\cee}{2- \bar\alpha} \rt)^{\frac{2- \bar\alpha}{4- \bar\alpha}} \,.
		\end{equation}
		We consider first the hypothesis \ref{Cov:bounded}. Since $\gamma$ is continuous, for any $\epsilon>0$, there is $\delta$ such that whenever $|z|\leq \delta\wedge r_0$, $\gamma(z) \geq \gamma(0)-\epsilon$. Hence,
		\begin{eqnarray*}
		\EE_B Z_m(n) \geq \exp \left\{ \lambda \sqrt{n} \left[ m(m-1) t \left( \gamma(0)-\epsilon \right) \right]^{1/2} \right\} \PP \left( \sup_{0 \leq s \leq t} |B_{0,t}^j (s)|\leq \delta\wedge r_0 \right)^m\,.
		\end{eqnarray*}
		Since as $n \to \infty$, $m \to \infty$ too, we have 
		\begin{eqnarray*}
		 \liminf_{n \to \infty} m^{-1} n^{-1/2} \log \EE Z_m(n) \geq \lambda\sqrt{ t (\gamma(0)-\epsilon)}\,,
		\end{eqnarray*}
		which proves \eqref{lim:ZmR} under \ref{Cov:bounded}.  

		Assume now that \ref{Cov:Dalang} holds.
		We put $t_n=t^{1-a}n^a$	so that $\epsilon_n=\epsilon\frac t{t_n}$.
		The Brownian motion scaling and the relation \eqref{eqn:scalgamma} yield 
		\begin{align*}
			\int_0^t \gamma_{\epsilon_n}(B^j_{0,t}(s)-B^k_{0,t}(s))ds
			&=\frac{t}{t_n} \int_0^{t_n} \gamma_{\epsilon\frac{t}{t_n}}\lt(B^j_{0,t}(s\frac{t}{t_n})-B^k_{0,t}(s\frac{t}{t_n})\rt)ds
			\\&\overset{\mathrm{law}}{=}\frac{t}{t_n}\int_0^{t_n} \gamma_{\epsilon\frac{t}{t_n}}\lt(\sqrt{\frac{t}{t_n}}(B^j_{0,t_n}(s)-B^k_{0,t_n}(s))\rt)ds
			\\&=\lt(\frac t{t_n} \rt)^{1- \frac\alpha2}\int_0^{t_n}\gamma_ \epsilon(B^j_{0,t_n}(s)-B^k_{0,t_n}(s))ds\,.
		\end{align*}
		It follows that
		\begin{align*}
			\EE Z_m(n)=\EE\lt[\exp\lt\{\lambda\lt(t_n \sum_{j,k=1}^m\int_0^{t_n}\gamma_ \epsilon(B^j_{0,t_n}(s)-B^k_{0,t_n}(s))ds \rt)^{\frac12} \rt\};\min_{1\le j\le m}\tau^j_D\ge t_n\rt]\,,
		\end{align*}
		where 
		\begin{equation*}
		  	\tau_D^j=\inf\lt\{s\ge 0:B^j(s)\not\in  D \rt\}\,.
		\end{equation*}
		Let $K_ \epsilon$ be the function defined by 
		\begin{equation*}
			K_ \epsilon(x)=(2 \pi)^{-\ell}\int_{\RR^\ell}e^{i \xi\cdot x-\frac \epsilon2|\xi|^2}\sqrt{\mu(\xi)}d \xi
		\end{equation*}
		so that
		\begin{equation}\label{eqn:gKK}
			\gamma_ \epsilon(x)=\int_{\RR^\ell}K_ \epsilon (y)K_ \epsilon(x-y)dy\,.
		\end{equation}
		Hence, we can write
		\begin{align*}
			\lt(t_n\sum_{j,k=1}^m\int_0^{t_n}\gamma_ \epsilon(B^j_{0,t_n}(s)-B^k_{0,t_n}(s))ds \rt)^{\frac12}
			=\lt( t_n\int_0^{t_n}\int_{\RR^\ell}\lt|\sum_{j=1}^m K_ \epsilon(x-B^j_{0,t_n}(s)) \rt|^2 dxds \rt)^{\frac12}\,.
		\end{align*}
		Let $\mathcal D$ be the set of compactly supported continuous functions on $\RR^\ell$ with unit $L^2(\RR^{\ell})$-norm. For every $f\in \mathcal D$, applying Cauchy-Schwarz inequality, we see that the right-hand side in the equation above is at least
		\begin{equation*}
			\sum_{j=1}^m\int_0^{t_n}\int_{\RR^\ell}f(x) K_ \epsilon\lt(x-B^j_{0,t_n}(s)\rt)dxds=\sum_{j=1}^m\int_0^{t_n}\bar f_ \epsilon\lt(B^j_{0,t_n}(s)\rt)ds\,,
		\end{equation*}
		where we have set
		\begin{equation*}
			\bar f_ \epsilon(x)=\int_{\RR^\ell}f(y)K_ \epsilon(y-x)dy\,.
		\end{equation*}
		Using independency of Brownian motions, we obtain
		\begin{equation*}
			\EE_B Z_m(n)\ge \lt(\EE_B\lt[\exp\lt\{\lambda\int_0^{t_n}\bar f_ \epsilon\lt(B_{0,t_n}(s)\rt)ds \rt\};\tau_D\ge t_n \rt] \rt)^m\,,
		\end{equation*}
		where $\tau_D := \inf\{s \geq 0: B(s) \notin D\}$. 
		Applying Lemma \ref{lem:expBB} we obtain
		\begin{align*}
			\liminf_{n\to\infty}\frac1{mt_n}\log\EE_B Z_m(n)\ge \sup_{g\in\cgg_D}\lt\{\lambda\int_D \bar f_ \epsilon(x)g^2(x)dx-\frac12\int_D|\nabla g(x)|^2dx \rt\}\,.
		\end{align*}
		We now let $D\uparrow\RR^\ell$ to get
		\begin{equation*}
			\liminf_{D\uparrow\RR^\ell}\liminf_{n\to\infty}\frac1{mn^a }\log\EE_B Z_m(n)\ge t^{1-a} \sup_{g\in\cgg}\lt\{\lambda\int_{\RR^\ell} \bar f_ \epsilon(x)g^2(x)dx-\frac12\int_{\RR^\ell}|\nabla g(x)|^2dx \rt\}\,.
		\end{equation*} 
		We now link the variation on the right-hand side with $\cmm(\gamma)$ by observing that
		\begin{align}\label{tmp:fgm}
			\sup_{f \in \mathcal{D}}\sup_{g\in\cgg}\lt\{\lambda\int_{\RR^\ell} \bar f_ \epsilon(x)g^2(x)dx-\frac12\int_{\RR^\ell}|\nabla g(x)|^2dx \rt\}=\cmm(\lambda^2\gamma_ \epsilon)\,.
		\end{align} 
		Indeed, for each fixed $g\in\cgg$, applying Fubini's theorem, Hahn-Banach theorem and \eqref{eqn:gKK}, we have
		\begin{align*}
			\sup_{f\in \mathcal D} \int_{\RR^\ell}\bar f_ \epsilon(x)g^2(x)dx
			&=\sup_{f\in \mathcal D}\int_{\RR^\ell}f(y)\int_{\RR^\ell}K_ \epsilon(y-x)g^2(x)dxdy
			\\&=\lt(\int_{\RR^\ell}\lt|\int_{\RR^\ell}K_ \epsilon(y-x)g^2(x)dx \rt|^2dy \rt)^{\frac12}
			\\&=\lt(\int_{\RR^\ell}\int_{\RR^\ell}\gamma_ \epsilon(x-y)g^2(x)g^2(y)dxdy \rt)^{\frac12}\,.
		\end{align*}
		This leads us the identity \eqref{tmp:fgm}.
		We can send $\epsilon\downarrow0$ and apply Proposition \ref{prop:ME} to obtain \eqref{lim:ZmR} under hypothesis \ref{Cov:Dalang}.
		
		\textbf{Step 3:} Combining the inequalities \eqref{tmp:step2} and \eqref{lim:ZmR}  together, we have for every $\lambda\in(0,\sqrt{2\ell})$
		\begin{equation*}
			\liminf_{n\to\infty}\sup_{j =1, \dots, N}\K_ {\epsilon_n}(x_0,y_j)\ge \frac{4- \bar\alpha}4 \lambda^{\frac4{4- \bar\alpha}} \lt(\frac{2t\cee}{2- \bar\alpha} \rt)^{\frac{2- \bar\alpha}{4- \bar\alpha}}\,.
		\end{equation*}
		Finally we let $\lambda \to \sqrt{2\ell}^-$ to conclude the proof.
	\end{proof}
	We now provide the proof of \eqref{claim.eta}.
		\begin{lemma}\label{lem:H1}
			For every $\lambda\in(0,\sqrt{2\ell})$, we have
			\begin{equation}\label{eq: eta goto 0 a.s.}
				\lim_{n \to \infty}\eta_n^c(x_0) = 0 \quad \mathrm{a.s.}
			\end{equation}
			where we recall $\eta_n^c$ is defined in \eqref{eq:eta R}.
		\end{lemma}
				


		\begin{proof}
		Assuming first that \ref{Cov:bounded} holds.	
		We recall that $\epsilon_n=0$ in this case so that $\gamma_{\epsilon_n}=\gamma$.
			Let $\mathcal{B}$ be the $\sigma$-field generated by the Brownian motions $\{B^j\}_{1\leq j \leq m}$. First we will show that for any $0 < \rho < \frac12$, we can find $d>0$ sufficiently large so that on the event $\{\min_{1\le j\le m}\tau^j\ge t\}$, for every $z,z'\in B(0,e^n)$ with $|z -z'|\ge d$. 
			\begin{equation}\label{eq:cov leq var}
			\text{Cov} \left( \xi_m(x_0,z), \xi_m(x_0,z') \Big|\mathcal{B}\right) \leq \rho S_m^2\,.
			\end{equation}
			We recall that $d$ and $\tau^j$ are defined in \eqref{y.pts} and \eqref{def.tau1} respectively. 
			We choose and fix $\varkappa\in(0,1)$ such that 
			\begin{equation}
				\varkappa \gamma(0) \leq \frac{1}{2}\rho\inf_{|x|\le 2 r_0}\gamma(x) \,.
			\end{equation}
			Note that on the event $\{\min_{1\le j\le m}\tau^j\ge t\}$, we have $\sup_{s\le t,j\le m}|B^j_{0,t}(s)|\le r_0$.
			Then for every $j,k\le m$, 
			\begin{align*}
				\int_0^{\varkappa t}\gamma\left(B_{0,t}^j(t-s) - B_{0,t}^k(t-s)+ \frac{s}{t} (z-z')\right) ds 
				&\le \varkappa t \gamma(0)
				\\&\le \frac \rho2\int_0^t\gamma\left(B_{0,t}^j(t-s) - B_{0,t}^k(t-s)\right) ds \,.
			\end{align*}
			In addition, from \eqref{eqn:gamspec} and Riemann-Lebesgue lemma, $\lim_{x\to\infty} \gamma(x)=0$. Hence, when $s\in [\varkappa t , t]$, we can choose $d$ large enough such that whenever $|y|\le 2r_0$ and $|z-z'|\ge d$
			\begin{equation*}
				\gamma(y+\frac st (z-z') )\le \frac \rho2 \gamma(y)	\,.
			\end{equation*} 
			In particular, for every $|z-z'|\ge d$ we have
			\begin{equation}
			\gamma\left(B_{0,t}^j(t-s) - B_{0,t}^k(t-s)+ \frac{s}{t} (z-z')\right) \leq \frac{\rho}{2} \gamma\left(B_{0,t}^j(t-s) - B_{0,t}^k(t-s)\right)\,.
			\end{equation}
			It follows that
			\begin{align*}
				&\text{Cov}\left( \xi_m(t,z), \xi_m(t,z') \Big| \mathcal B\right)\\
				&=\sum_{j, k=1}^m \int_0^t \gamma\left( B_{0,t}^j(t-s) -  B_{0,t}^k(t-s) + \frac{s}{t} (z-z')\right) ds
				\\&\le \rho \sum_{j, k=1}^m \int_0^t \gamma\left( B_{0,t}^j(t-s) -  B_{0,t}^k(t-s) \right) ds\,,
			\end{align*}
			which verifies \eqref{eq:cov leq var}.

			Since $\lambda < \sqrt{2\ell}$, we can choose $\kappa, \rho \in(0,\frac12)$ sufficiently small so 
			\begin{equation}
			\frac{(1+2\rho)(\lambda+\kappa)^2}{2} < \ell \quad \text{and} \quad \frac{\kappa^2}{4\rho} > \ell+1\,.
			\end{equation}
			Let us now recall Lemma 4.2 in \cite{MR3178468}. For a mean zero n-dimensional Gaussian vector $(\xi_1, \cdots, \xi_n)$ with identically distributed components, 
			\begin{equation}
			\max_{i \neq j} \frac{|\text{Cov} (\xi_i, \xi_j)|}{ \text{Var}(\xi_1)} \leq \rho< \frac{1}{2}
			\end{equation}
			and for any $A,B >0$, we have
			\begin{equation}
			\PP \left\{ \max_{k \leq n} \xi_k \leq A \right\} \leq \left( \PP \left\{ \xi_1 \leq \sqrt{1+2\rho} (A+B)\right\}\right)^n + \PP \left \{ U \geq B/\sqrt{2\rho \text{Var} (\xi_1)} \right\}
			\end{equation}
			where $U$ is a standard normal random variable. Applying this inequality conditionally with $A=\lambda S_m(t) \sqrt{n}$ and $B= \kappa S_m \sqrt{n}$, we have for sufficiently large $n$,
			\begin{align*}
			&\PP \left\{ \max_{j=1,\dots,N} \xi_m(x_0,y_j) < \lambda \sqrt{n} S_m \Big| \mathcal{B}   \right\}\\
			&\leq \left(\PP \left\{ U \leq \sqrt{1+2\rho} (\lambda+\kappa) \sqrt{n}\right\} \right)^{N} + \PP \left\{ U \geq \frac\kappa{\sqrt{2\rho}} \sqrt{n} \right\}\\
			&\leq \exp \left\{ - (1+o(1)) C e^{vn} \right\} + e^{-(\ell+1)n} \leq C e^{-(\ell+1)n}\,,
			\end{align*}
			where $v>0$ is independent of $n$. Now for any $\theta>0$, this yields
			\begin{align*}
				\PP(\eta_n^c(x_0)\ge \theta)&\le \theta^{-1}\EE \eta_n^c(x_0) 	
				\\&=(\theta\EE Z_m(n))^{-1}\EE\lt[Z_m(n)\PP\left\{ \max_{j=1,\dots,N} \xi_m(x_0,y_j) < \lambda \sqrt{n} S_m \Big| \mathcal{B}   \right\}\rt]
				\\&\lesssim Ce^{-(\ell+1)n}\,.
			\end{align*} 
			An application of Borel-Cantelli lemma yields (\ref{eq: eta goto 0 a.s.}) under hypothesis \ref{Cov:bounded}.

			We now consider the hypothesis \ref{Cov:Dalang}.
			The argument is similar to the previous case. There is, however, an additional scaling procedure. Recall that $\mathcal{B}$ is the $\sigma$-field generated by the Brownian motions $\{B^j\}_{1\leq j \leq m}$. We choose $d=1$. It suffices to prove \eqref{eq:cov leq var} on the event $\{ \min_{0\le j\le m} \tau^j\ge t \}$, for any $|z-z'|\ge1$.
			Indeed, we have
			\begin{align*}
				\text{Cov}\left( \xi_m(x_0,z), \xi_m(x_0,z')\Big| \mathcal B \right)
				=\sum_{j, k=1}^m \int_0^t  \gamma_{\epsilon_n} \left( B_{0,t}^j(t-s) -  B_{0,t}^k(t-s) + \frac{s}{t} (z-z')\right) ds\,.
			\end{align*}
			For every $j,k\le m$, using the scaling relation \eqref{eqn:scalgamma},
			we can write
			\begin{multline*}
				\gamma_{\epsilon_n}\left(B_{0,t}^j(t-s) - B_{0,t}^k(t-s)+ \frac{s}{t} (z-z')\right)
				\\=\epsilon_n^{-\frac \alpha2} \gamma_{1}\left(\epsilon_n^{-\frac12}(B_{0,t}^j(t-s) - B_{0,t}^k(t-s))+ \epsilon_n^{-\frac12}\frac{s}{t} (z-z')\right) \,.
			\end{multline*}
			We now choose and fix $\theta>0$ such that 
			\begin{equation}
			\theta   \leq\frac{\rho}{2\gamma_1(0)}\inf_{x\in \epsilon^{-1/2} D} \gamma_1(x)\,,
			\end{equation}
			this is always possible since $\gamma_1=p_{2}*\gamma$ is a strictly positive function.
			It follows that
			\begin{align*}
				&\epsilon_n^{-\frac \alpha2}\int_0^{\theta t} \gamma_{1}\left(\epsilon_n^{-\frac12}(B_{0,t}^j(t-s) - B_{0,t}^k(t-s))+ \epsilon_n^{-\frac12}\frac{s}{t} (z-z')\right)ds
				\\&\le\epsilon_n^{-\frac \alpha2} \theta t \gamma_1(0)
				\\&\le \frac \rho2\epsilon_n^{-\frac \alpha2}\int_0^t\gamma_{1}\left(\epsilon_n^{-\frac12}(B_{0,t}^j(t-s) - B_{0,t}^k(t-s))\right)ds 
				\\&= \frac \rho2\int_0^t\gamma_{\epsilon_n}\left(B_{0,t}^j(t-s) - B_{0,t}^k(t-s)\right)ds \,.
			\end{align*}
			In addition, on the event $\{\min_{0\le j\le m} \tau^j\ge t \}$, $\epsilon_n^{-\frac12}(B_{0,t}^j(t-s)-B_{0,t}^k(t-s))$ belongs to $2\epsilon^{-\frac12}D$ for all $s\in[0,t]$. Hence, for every $s\in[\theta t,t]$ and $|z-z'|\ge1$, we have
			\begin{equation*}
				\lt|\epsilon_n^{-\frac12}(B_{0,t}^j(t-s)-B_{0,t}^k(t-s))+ \epsilon_n^{-\frac12}\frac{s}{t} (z-z')\rt|\ge \theta  \epsilon_n^{-\frac12}- 2 \epsilon^{-\frac12} \mathrm{diag}(D)\,.
			\end{equation*}
			We note that from Riemann-Lebesgue lemma, $\lim_{x\to\infty} \gamma_1(x)=0$.
			Hence, whenever $n$ is sufficiently large, 
			\begin{equation*}
				\gamma_1(y)\le \frac \rho2\inf_{x\in \epsilon^{-1/2} D}\gamma_1(x) 	
			\end{equation*}
			for all $|y|\ge \theta  \epsilon_n^{-\frac12}- 2 \epsilon^{-\frac12} \mathrm{diag}(D)$. It follows that for every $z,z'$ with $|z-z'|\ge1$,
			\begin{align*}
				&\epsilon_n^{-\frac \alpha2}\int_{\theta t}^t \gamma_{1}\left(\epsilon_n^{-\frac12}(B_{0,t}^j(t-s) - B_{0,t}^k(t-s))+ \epsilon_n^{-\frac12}\frac{s}{t} (z-z')\right) ds
				\\&\le\epsilon_n^{-\frac \alpha2}\frac \rho2  \int_{\theta t}^t\gamma_{1}\left(\epsilon_n^{-\frac12}(B_{0,t}^j(t-s) - B_{0,t}^k(t-s))\right) ds
				\\&\le \frac{\rho}{2} \int_{0}^t\gamma_{\epsilon_n}\left(B_{0,t}^j(t-s) - B_{0,t}^k(t-s)\right)ds\,.
			\end{align*}
			Upon combining these estimates, we arrive at \eqref{eq:cov leq var}, which in turn, implies \eqref{eq: eta goto 0 a.s.}. 
		\end{proof}
\subsection{Proofs} 
\label{sec:proof}
The Theorems \ref{thm:H1Z}, \ref{thm:H23Z} and \ref{thm:uep} follow from the asymptotic results from the previous two subsections. Indeed, Theorem \ref{thm:H1Z} follows by combining the upper bound in Theorem \ref{thm:Z up} and the lower bound Theorem \ref{thm:low}. To obtain Theorem \ref{thm:H23Z}, we first observe that from \eqref{rep:u}, 
\begin{equation}\label{eq: u less Z}
\frac{u(t,y)}{p_t*u_0(y)}\leq \sup_{x \in \supp u_0} \frac{\Z(x; t,y)}{p_t(x-y)}\,.
\end{equation}
Then, an application of Theorem \ref{thm:Z up} yields the result. For Theorem \ref{thm:uep}, the upper bound of \eqref{lim:H2epup} follows from Remark \ref{rmk: Z epsilon up} and the bound \eqref{eq: u less Z} with $u,\Z$ replaced respectively by $u_{\epsilon}, \Z_ \epsilon$,  together with the obvious fact that $\mathcal{E}_H(\gamma_{\epsilon})\leq \mathcal{E}_H(\gamma)$, see \eqref{eqn:cee}.  The lower bound of \eqref{lim:H2ep} is immediate from Theorem \ref{thm:low}.

\section*{Acknowledgment} 
The project was initiated while both authors are visiting the Mathematical Sciences Research Institute in Berkeley, California, during the Fall 2015 semester.  The first named author thanks Professor Davar Khoshnevisan for stimulating discussions and encouragement. The second named author thanks PIMS for its support through the Postdoctoral Training Centre in Stochastics. Both authors thank Professor Yaozhong Hu and Professor David Nualart for support and encouragement. 

\bibliography{../bib}
\end{document}